\input{preambule.pwz}
\def \cJ {\mathcal{J}}
\documentclass{amsart}
\usepackage[OT1]{fontenc}
\usepackage{amsthm,amsmath,amsfonts,amssymb}
\usepackage[all,cmtip]{xy}
\usepackage[pdftex]{graphicx}
\usepackage[colorlinks]{hyperref}
\input{definitions.pwz}

\begin{document}

\title[Approximation of the marginal likelihood for tree models]{An asymptotic approximation of the marginal likelihood for general Markov models}
\author{{Piotr} {Zwiernik}}
\markboth{P. Zwiernik et al.}{Tree models for binary data}
\address{Piotr Zwiernik\\University of Warwick\\Department of Statistics\\CV7AL, Coventry, UK.}
\email{p.w.zwiernik@warwick.ac.uk}

\date{}

\begin{abstract}  The standard Bayesian Information Criterion (BIC) is derived under regularity conditions which are not always satisfied by the graphical models with hidden variables. In this paper we derive the BIC score for Bayesian networks in the case of binary data and when the underlying graph is a rooted tree and all the inner nodes represent hidden variables. This provides a direct generalization of a similar formula given by Rusakov and Geiger for naive Bayes models. The main tool used in this paper is a connection between asymptotic approximation of Laplace integrals and the real log-canonical threshold. \end{abstract}

\maketitle

\section{Introduction}

A key step in Bayesian approach to learning graphical models is to compute the \textit{marginal likelihood} of the data, i.e. the \textit{observed likelihood function} averaged over the parameters with respect to the prior distribution. Given a fully observed system the theory of  graphical models provides a simple way to obtain the marginal likelihood (see e.g. \cite{cooper1992bayesian}, \cite{heckerman1995lbn}). However, when some of the variables in the system are \textit{hidden} (i.e. never observed), the exact determination of the marginal likelihood is typically intractable (e.g. \cite{chickering1997efficient}, \cite{cooper1992bayesian}). Therefore, there is a need to develop efficient approximate techniques for computing the marginal likelihood. 

In this paper we focus on large sample approximations for the marginal likelihood called the \textit{BIC approximation}. Let $\cM$ be a parametric discrete model and $X^{(N)}=X^1,\ldots, X^N$ be a random sample from $\cM$ of size $N$. By $Z(N)$ we denote the marginal likelihood and by $L(\theta;X^{(N)},\cM)=\P(X^{(N)}|\cM,\theta)$ the observed likelihood function. Thus
\begin{equation}\label{eq:ZN0}
Z(N)=\P(X^{(N)}|\cM)=\int_{\Theta} L(\theta;X^{(N)},\cM)\varphi(\theta){\rm d}\theta,
\end{equation}
where $\theta$ denotes the model parameters, $\Theta\subseteq \R^{d}$ is the parameter space, and $\varphi(\theta)$ is a prior distribution on $\Theta$ given model $\cM$.

In statistical theory to obtain the BIC approximation we usually require that the observed likelihood is maximized over a unique point in the interior of the parameter space. For the class of problems for which this assumption is satisfied Schwarz \cite{schwarz1978edm} showed that as $N\rightarrow \infty$
\begin{equation}\label{eq:bicfinite}
\log Z(N)=\hat{\ell}_{N}-\frac{d}{2}\log N+O(1),
\end{equation}
where $\hat{\ell}_{N}$ is the maximum value of the log-likelihood and $d=\dim\Theta$. The same approximation works if the observed likelihood is maximized over a finite number of points. Geometrically, for large sample sizes function $Z(N)$ concentrates around the maxima (see Figure \ref{fig:finite}).
\begin{figure}[t]
\includegraphics[scale=0.3]{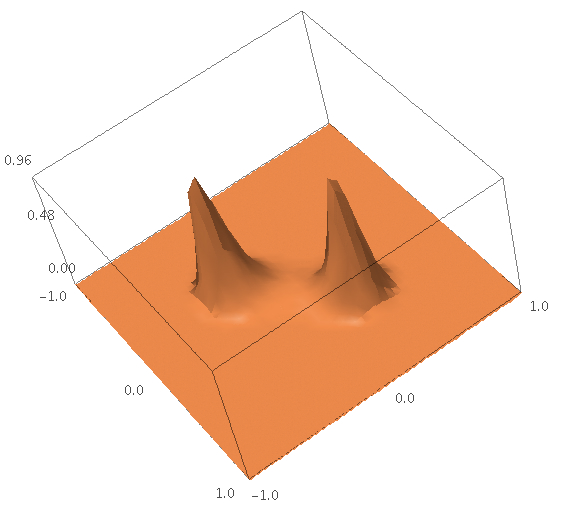}
\caption{The case when the observed likelihood is maximized over a finite number of points.}\label{fig:finite}
\end{figure}
This enables us to apply the Laplace approximation locally in the neighborhood of each maximum. 

It can be proved (see Proposition \ref{prop:regulargeneral}) that the above  formula can be generalized for the case when the set over which the likelihood is maximized forms a sufficiently regular compact subset of the ambient space (see Figure \ref{fig:smooth}). We denote this subset by $\widehat{\Theta}$. In this case as $N\rightarrow \infty$
\begin{equation}\label{eq:bicsmotth}
\log Z(N)=\hat{\ell}_{N}-\frac{d-d'}{2}\log N+O(1),
\end{equation}
where $d'=\dim \widehat{\Theta}$. Note that in our case $\widehat{\Theta}$ is a set of zeros of a real analytic function. Therefore it will be always a semi-analytic set, i.e. given by $\{g_{1}(\theta)\geq 0, \ldots, g_{r}(\theta)\geq 0\}$, where $g_{i}$ are all analytic functions. It follows that the dimension is well defined (see \cite[Remark 2.12]{bierstone1988semianalytic}). 
\begin{figure}[t]
\includegraphics[scale=0.3]{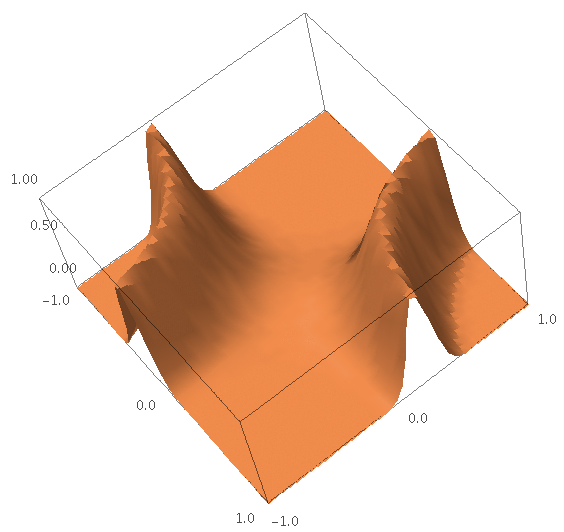}
\caption{The case when the observed likelihood is maximized over an infinite but smooth subset given by $xy=1$ for $x,y\in[-1,1]$.}\label{fig:smooth}
\end{figure}

In the case of models with hidden variables for some data sets the locus of the points maximizing the likelihood may not be sufficiently regular. In this case the likelihood will have a different asymptotic behavior around the singular points and relatively more mass of the marginal likelihood integral will be related to neighborhoods of singular points (see Figure \ref{fig:singular}).
\begin{figure}
\includegraphics[scale=0.3]{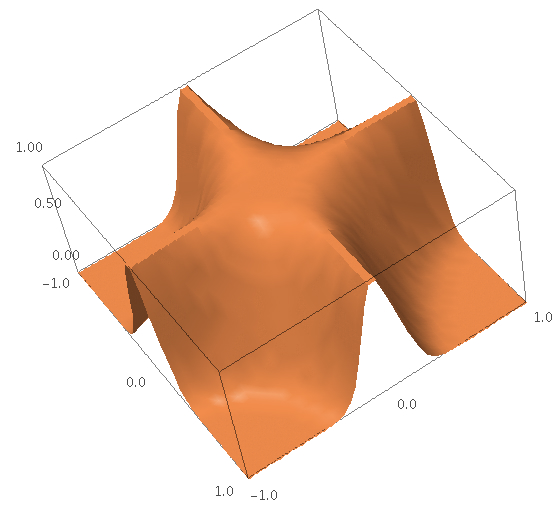}
\caption{The case when the observed likelihood is maximized over a singular subset given by $xy=0$ for $x,y\in[-1,1]$.}\label{fig:singular}
\end{figure}
For these points we cannot use the Laplace approximation. Nevertheless the computation of the BIC approximation is still possible by using the results of Watanabe \cite{watanabe_book} and linking this to some earlier works of Arnold, Varchenko and collaborators (see e.g. \cite{arnold1985singularities}). This approximation will differ from the standard BIC formula. First, the coefficient of $\log N$ can be different from $-\frac{d-d'}{2}$. Second, we sometimes obtain an additional $\log \log N$ term affecting the asymptotics (see Theorem \ref{th:watanabe}).

In this paper we consider an important model class with large number of hidden variables called the general Markov model. This model class is extensively used in phylogenetics (e.q. \cite[Chapter 8]{semple2003pol}) and in casuality analysis (e.g. \cite{pearl_tarsi86}). The general Markov model is a Bayesian network on a tree. Thus let $T=(V,E)$ be a tree with the vertex set $V$ and the edge set $E$. Let $T^{r}$ denote a \textit{ tree  rooted in $r$}, i.e. a tree with one distinguished vertex $r$ and all the edges directed away from $r$. Let $Y=(Y_{v})_{v\in V}$ be a collection of binary random variables indexed by the set of vertices of $T$. We assume that all the inner nodes represent hidden random variables. Hence the general Markov model, denoted by $\cM_{T}$, is a family of marginal distributions over the vector of random variables representing the leaves of $T$. 

A surprising fact proved in this paper is that, given the sample proportions lie within the model class, the zeros in the sample covariance matrix of the vector of observed random variables completely determine the asymptotics for the model $\cM_{T}$. In this paper following \cite{rusakov2006ams} we always assume:
\begin{description}
\item[(A1)]{The prior distribution $\varphi:\Theta\rightarrow \R$ is  strictly positive, bounded and smooth on $\Theta$.}
\item[(A2)]{There exists $N_{0}$ such that $\hat{p}^{(N)}=\hat{p}\in\cM$ for all $N>N_{0}$ and $\hat{p}$ has positive entries.}
\end{description}
For a given sample covariance matrix $\widehat{\Sigma}=[\hat{\mu}_{ij}]$ let $l_{2}$ denote the number of inner nodes $v$ of $T$ such that for each triple $i,j,k$ of leaves separated in $T$ by $v$ we have $\hat{\mu}_{ij}\hat{\mu}_{ik}\hat{\mu}_{jk}=0$ but there exist leaves $i,j$ separated by $v$ such that $\hat{\mu}_{ij}\neq 0$. Here we say that two nodes $u,v$ of $T$ are \textit{separated} by another node $w$ if $w$ lies on the unique path between $u$ and $v$. We define a \textit{degenerate node} as an inner nodes $v$ such that for any two leaves $i,j$ separated by $v$ we have $\hat{\mu}_{ij}=0$. All other nodes are called \textit{nondegenerate}. We denote by $n_{e}$ the number of edges of $T$ and by $n_{v}$ the number of its nodes. 
\begin{thm}\label{prop:regularcase} Let $T^{r}$ be a rooted tree with $n$ {leaves} representing binary random variables $X_{1},\ldots, X_{n}$ and assume that their joint distribution lies in the general Markov model $\cM_{T}$. Let $X^{(N)}$ be $N$ independent realizations of this vector and let $\hat{p}^{(N)}$ the corresponding sample proportions. With assumptions (A1) and (A2), if there are no degenerate nodes then as $N\rightarrow \infty$
$$
\log Z(N)=\hat{\ell}_{N}-\frac{n_{v}+n_{e}-2l_{2}}{2}\log N+O(1),
$$
where $\hat{\ell}_{N}$ is the maximum log-likelihood value. 
\end{thm}
In general if there are degenerate nodes the computations of the BIC approximation are much harder because the likelihood in this case maximizes over a singular subset of the parameter space. In this paper we  obtain a closed form formula for the BIC approximation in the case of \textit{trivalent trees}, i.e. the trees such that each inner node has valency three. This is provided in Theorem \ref{th:main} which together with Theorem \ref{prop:regularcase} are the main results of this paper. 

Let $l_{3}$ denote the number of inner nodes $v$ such that for every $i,j\in [n]$ such that the path between $i$ and $j$ crosses $v$ we have that $\hat{\mu}_{ij}\neq 0$. 
\begin{thm}\label{th:main}
Let $T=(V,E)$ be a rooted trivalent tree with $n\geq 3$ leaves and root $r$. Let $X=(X_{1},\ldots, X_{n})$ be a binary random vector representing the leaves of $T$ and assume that the joint distribution of $X$ lies in $\cM_{T}$. Let $X^{(N)}$ be a random sample given by $N$ independent realization of $X$ and $\hat{p}^{(N)}$ the corresponding sample proportions. With assumptions (A1) and (A2) if $r$ is degenerate but all its neighbors are not, then  as $N\rightarrow\infty$
$$
\log Z(N)=\hat{\ell}_{N}-\frac{3n+l_{2}+5l_{3}-1}{4}\log N+O(1).
$$
In all other cases as $N\rightarrow\infty$
$$
\log Z(N)=\hat{\ell}_{N}-\frac{3n+l_{2}+5l_{3}}{4}\log N+c\log\log N+O(1),
$$
where $c\geq 0$. Moreover $c=0$ always if $r$ is nondegenerate or if $r$ and all its neighbors are degenerate.
\end{thm}
Following \cite{rusakov2006ams} the main method of  proof is to change the coordinates of the models so that the induced parameterization becomes simple. This gives us a much better insight into the model structure (see \cite{pwz2010-identifiability}, \cite{pwzsmith2009}). Since the  BIC approximation is invariant with respect to these changes the reparameterized problem still gives the solution to the original question.   The main analytical tool is the real log-canonical threshold (e.g. \cite{saito2007real}, \cite{watanabe_book}). This is an important geometric invariant which in certain cases can be computed in a relatively simple way using discrete geometry. The relevance of this invariant to the BIC approximation is given by Theorem \ref{th:watanabe}. Techniques developed in this paper can be applied to obtain the BIC approximation also in the non-trivalent case.


The paper is organized as follows. In Section \ref{sec:approx} we provide the theory of asymptotic approximation of the marginal likelihood integrals. This theory allows us to approximate marginal likelihood without the standard regularity assumptions. Theorem \ref{th:watanabe} links these concepts with the real log-canonical threshold which allows us to use simple algebraic arguments. In Section \ref{sec:tree-models} we define Bayesian networks on rooted trees. We also obtain some simple result on the BIC approximation in the case when the observed likelihood is maximised over a sufficiently smooth subset of the parameter space. This gives a simple proof of Theorem \ref{prop:regularcase}. The proof of Theorem \ref{th:main} is more technical and so divided it into three main steps. By Theorem \ref{th:watanabe} to obtain the asymptotic approximation we need to compute a certain real log-canonical threshold. In the first step, in Section \ref{sec:sing-case}, following \cite{shaowei_rlct} we introduce the concept of the real log-canonical threshold of an ideal. Theorem \ref{th:barI} reduces our computations to the real log-canonical threshold of an ideal induced by the parametrization of the given model. This result applies for general discrete statistical models. Theorem \ref{th:redtokappas} gives an additional reduction which can be obtained only for tree models. Second step is given in Section \ref{sec:splittrees} where we show that the computations can be reduced to two distinct cases. One of them is the case already considered  in Section \ref{sec:tree-models}. The second case is more complicated and requires to use the method of Newton diagrams. We analyze this case in Section \ref{sec:zerocase}. Finally in  Section \ref{sec:last-proof} we combine all the results.

\section{Asymptotics of marginal likelihood integrals}\label{sec:approx}

In this section we introduce the real log-canonical threshold and link it with the problem of asymptotic approximation of Laplace integrals. We present how this enables us to obtain the BIC approximation in the case of a general class of statistical models.

\subsection{The real log-canonical threshold}\label{sec:rlct}
Given $\theta_{0}\in \R^{d}$, let $\cA_{\theta_{0}}(\R^{d})$ be the ring of real-valued functions $f:\R^{d}\rightarrow \R$ that are analytic at $\theta_{0}$. Given a subset $\Theta \subset\R^{d}$, let $\cA_{\Theta}(\R^{d})$ be the ring of real functions analytic at each point $\theta_{0}\in\Theta$. If $f\in\cA_{\Theta}(\R^{d})$, then for every $\theta_{0}\in \Theta$, $f$ can be locally represented as a power series centered at $\theta_{0}$. Denote by $\cA^{\geq}_{\Theta}(\R^{d})$ the subset of $\cA_{\Theta}(\R^{d})$ consisting of all non-negative functions.  Usually the ambient space is clear from the context and in this case we omit it in our notation writing $\cA_{\theta_{0}}$ and so on. We assume that $\Theta\subseteq\R^{d}$ is a compact and \textit{semianalytic set} of dimension $d$, i.e. $\Theta=\{x\in\R^{d}:g_{1}(x)\geq 0,\ldots, g_{l}(x)\geq 0\}$, where $g_{i}$ are analytic functions. 
\begin{defn}[The real log-canonical threshold]\label{def:rlct}
Given a compact semianalytic set $\Theta\subseteq\R^{d}$ such that $\dim \Theta=d$, a real analytic function $f\in\cA^{\geq}_{\Theta}(\R^{d})$ and a smooth positive function $\varphi:\R^{d}\rightarrow \R$, consider the \textit{zeta function} defined as
\begin{equation}\label{eq:zeta}
\zeta(z)=\int_{\Theta} f(\theta)^{-z} \varphi(\theta){\rm d} \theta.\end{equation}
This function  is extended to a meromorphic function in $z$ on the entire complex line (c.f. Theorem 2.4 in  \cite{watanabe_book}). The \textit{real log-canonical threshold of $f$} denoted by ${\rm rlct}_{\Theta}(f;\varphi)$ is the smallest pole of $\zeta(z)$. By ${\rm mult}_{\Theta}(f;\varphi)$ we denote the multiplicity of this pole. By convention if $\zeta(z)$ has no poles then ${\rm rlct}_{\Theta}(f;\varphi)=\infty$ and ${\rm mult}_{\Theta}(f;\varphi)=d$. If $\varphi(\theta)\equiv 1$ then we omit $\varphi$ in the notation writing ${\rm rclt}_{\Theta}(f)$ and ${\rm mult}_{\Theta}(f)$. Define ${\rm RLCT}_{\Theta}(f;\varphi)$ to be the pair $({\rm rlct}_{\Theta}(f;\varphi), {\rm mult}_{\Theta}(f;\varphi))$, and we order these pairs so that $(r_{1},m_{1})>(r_{2},m_{2})$ if $r_{1}>r_{2}$, or $r_{1}=r_{2}$ and $m_{1}<m_{2}$. 
\end{defn}
To show that the real log-canonical threshold is well defined we need to show that if $\zeta(z)$ has poles then the minimal pole always exists. This is easy to see if $f$ and $\varphi$ are monomial functions as in the example below.

\begin{exmp}\label{ex:monomial}
Let $f:\R^{d}\rightarrow \R$ such that $f(x)=x^{2u}=x_{1}^{2u_{1}}\cdots x_{d}^{2u_{d}}$ and $\varphi(x)=x^{h}=x_{1}^{h_{1}}\cdots x_{d}^{h_{d}}$ where $u,h\in \N^{d}$. If $\Theta=[-\e,\e]^{d}$ is an $\e$-box around the origin in $\R^{d}$ then the zeta function in (\ref{eq:zeta}) becomes
$$
\zeta(z)=\int_{\Theta} \prod_{i=1}^{d}x_{i}^{h_{i}-2u_{i}z}{\rm d}x=C(\e)\prod_{i=1}^{d} \frac{1}{1+h_{i}-2u_{i}z},
$$ 
where $C(\e)$ is a constant depending on $\e$. Hence the poles of $\zeta(z)$ are positive rational functions given by $\frac{1+h_{i}}{2u_{i}}$ for $i=1,\ldots, d$. In this case the smallest pole is given by the minimal of these numbers and the multiplicity is given by the number of times the minimum occurred.  
\end{exmp}
The computation of  poles and their multiplicities of $\zeta(z)$ is linked to the asymptotic expansion of the Laplace integral 
\begin{equation}\label{eq:laplace2}
I(N)=\int_{\Theta}e^{-Nf(\theta)}\varphi(\theta){\rm d}\theta,
\end{equation}
for large values of the parameter $N$. This theory was independently developed in Section 7.2 in \cite{arnold1985singularities} and Section 2.4 and Section 6.2 in \cite{watanabe_book}. The following theorem gives this relation. In Section \ref{sec:marg-like} we show how it can be used to obtain the BIC approximation under a fairly general statistical setting which will be later specialized to general Markov models for binary data.
\begin{thm}
\label{th:watanabe}Let $\Theta$ be a compact semianalytic subset of $\R^{d}$ and $f\in\cA_{\Theta}^{\geq}(\R^{d})$. Let $I(N)$ be defined as in (\ref{eq:laplace2}). Then as $N\rightarrow \infty$
$$
\log I(N)=-{\rm rlct}_{\Theta}(f;\varphi) \log N + ({\rm mult}_{\Theta}(f;\varphi)-1)\log \log N + O(1).
$$
\end{thm}
\begin{proof}
This is a special case of Theorem 4.2 in \cite{shaowei_rlct} such that $r=1$ and $f_{1}=\sqrt{f}$.
\end{proof}

To compute the real log-canonical threshold we split integral in (\ref{eq:zeta}) into a sum of finitely many integrals over small neighbourhoods $\Theta_{0}$ of some points $\theta_{0}\in \Theta$. We can always do this using a partition of unity since $\Theta$ is compact (see e.g. \textsection 16, \cite{munkres_manif}). For each of the local integrals we use Hironaka's theorem stated below to reduce it to a locally monomial case which can be easily dealt with as  in Example \ref{ex:monomial}. The version of Hironaka's theorem  we are going to use in this paper was first formulated in \cite{atiyah1970resolution}.
\begin{thm}[Hironaka's theorem]\label{th:hironaka}
Let $f:\R^{d}\rightarrow \R$ be a real analytic function in the neigborhood of the origin  such that $f(0)=0$. Then there exists a neighborhood of the origin $W$ and a proper real analytic map $\pi: U\rightarrow W$ where $U$ is a $d$-dimensional real analytic manifold such that
\begin{enumerate}
\item The map $\pi$ is an isomorphism between $U\setminus U_{0}$ and $W\setminus W_{0}$, where $W_{0}=\{x\in W: f(x)=0\}$ and $U_{0}=\{u\in U: f(\pi(u))=0\}$.
\item For an arbitrary point $P\in U_{0}$, there is a local coordinate system $u=(u_{1},\ldots, u_{d})$ of $U$ in which $P$ is the origin and 
$$
f(\pi(u))=a(u) u_{1}^{r_{1}}\cdots u_{d}^{r_{d}},
$$
where $a(u)$ is a nowhere vanishing function on this local chart and $r_{1},\ldots, r_{d}$ are nonnegative integers, and the Jacobian determinant of $x=\pi(u)$ is 
$$
\pi'(u)=b(u)u_{1}^{h_{1}}\cdots u_{d}^{h_{d}},
$$
where again $b(u)\neq 0$ and $h_{1},\ldots, h_{d}$ are nonnegative integers. 
\end{enumerate}
Moreover $\pi$ can be always obtained as a composition of blow-ups along smooth centers. 
\end{thm}
For the construction of the blow-up see for example Section 3.5 in \cite{watanabe_book}. 

The local computations are performed as follows. Let $\theta_{0}\in\Theta$ and let $W_{0}$ be any sufficiently small open ball around $\theta_{0}$ in $\R^{d}$. Then, by Theorem 2.4 in \cite{watanabe_book}, ${\rm RLCT}_{ W_{0}}(f;\varphi)$ does not depend on the choice of $W_{0}$ and hence it is denoted by ${\rm RLCT}_{\theta_{0}}(f;\varphi)$. Formally for this local computation we consider $f$ centered at $\theta_{0}$, i.e. the function $f(\theta+\theta_{0})$. If $f(\theta_{0})\neq 0$ then ${\rm RLCT}_{W_{0}}(f;\varphi)=(\infty, d)$ and hence we can constrain only to points $\theta_{0}$ such that $f(\theta_{0})=0$. In this case by Hironaka's theorem  
$$
\int_{W_{0}}(f(\theta))^{-z}\varphi(\theta){\rm d}\theta=\int_{W} (f(\theta+\theta_{0}))^{-z}\varphi(\theta+\theta_{0}) {\rm d}\theta=\sum_{\beta} \int_{U_{\beta}} u^{h-2rz} c_{\beta}(u) {\rm d}u,
$$
where $W$ is the neighbourhood $W_{0}$ translated to the origin and the (finite) sum is over all local charts as in the theorem such that they cover $\pi^{-1}(W)$ and $c_{\beta}$ are nowhere vanishing functions on $U_{\beta}$. Then for each of the charts we do computations as in Example \ref{ex:monomial}. Consequently ${\rm rlct}_{\theta_{0}}(f;\varphi)=\min_{\beta}\min_{i} \frac{1+h_{i}}{2r_{i}}$ and $${\rm mult}_{\theta_{0}}(f;\varphi)=\max_{\beta}\#\{i\in\{1,\ldots,d\}:\, \frac{1+h_{i}}{2r_{i}}={\rm rlct}_{\theta_{0}}(f;\varphi)\}.$$ 
In particular ${\rm rltc}_{\Theta}(f;\varphi)$ is always a positive rational number and ${\rm mult}_{\Theta}(f;\varphi)$ is a nonnegative integer which shows that Definition \ref{def:rlct} makes sense. Moreover by Theorem 2.4 in \cite{watanabe_book} the real log-canonical threshold does not depend on the triple $(W,U,\pi)$.

The local computations give the answer to the global question since by \cite[Proposition 2.5]{shaowei_rlct} the set of pairs ${\rm RLCT}_{\Theta_{0}}(f;\varphi)$ for $\theta_{0}\in \Theta$ has a minimum and
\begin{equation}\label{eq:rlct-min}
{\rm RLCT}_{\Theta}(f;\varphi)=\min_{\theta_{0}\in\Theta}{\rm RLCT}_{\Theta_{0}}(f;\varphi),
\end{equation}
where $\Theta_{0}=W_{0}\cap \Theta$. For each $\theta_{0}\in \Theta$ to compute ${\rm RLCT}_{\Theta_{0}}(f;\varphi)$ we consider two cases. If $\theta_{0}$ lies in the interior of $\Theta$ then we can assume $\Theta_{0}=W_{0}$ and hence ${\rm RLCT}_{\Theta_{0}}(f;\varphi)={\rm RLCT}_{\theta_{0}}(f;\varphi)$. If $\theta_{0}\in{\rm bd}(\Theta)$, where ${\rm bd}(\Theta)$ denotes the boundary of $\Theta$, the computations may change significantly because the real log-canonical threshold depends on the boundary conditions (c.f. Example 2.7 in \cite{shaowei_rlct}). Nevertheless it can be showed that at least if there exists an open subset $U\subseteq \R^{d}$ such that $U\supset \Theta$ and $f\in\cA_{U}^{\geq}(\R^{d})$ then 
\begin{equation}\label{eq:rlct-constr}
{\rm RLCT}_{\Theta_{0}}(f)\geq {\rm RLCT}_{\theta_{0}}(f).
\end{equation}
For in this case 
$$
\int_{W_{0}}(f(\theta))^{-z}{\rm d}\theta=\int_{\Theta_{0}}(f(\theta))^{-z}{\rm d}\theta+\int_{W_{0}\setminus \Theta_{0}}(f(\theta))^{-z}{\rm d}\theta
$$
which implies that 
 $${\rm RLCT}_{\theta_{0}}(f)=\min\{{\rm RLCT}_{\Theta_{0}}(f),{\rm RLCT}_{W_{0}\setminus \Theta_{0}}(f)\}.$$

If $f\in\cA_{\Theta}^{\geq}(\R^{d})$, then let $\widehat{\Theta}:=f^{-1}(0)$. By definition \ref{def:rlct}, ${\rm RLCT}_{\theta_{0}}(f)=(\infty,d)$ for all $\theta_{0}\notin\widehat{\Theta}$ and hence we can restrict ourselves to points in $\widehat{\Theta}$. Therefore, whenever $\widehat{\Theta}\neq \emptyset$ we have\begin{equation}\label{eq:takemin}
{\rm RLCT}_{\Theta }(f)=\min_{\theta_{0}\in\widehat{\Theta} }{\rm RLCT}_{{\Theta}_{0}}(f).
\end{equation}

\begin{rem}
Note that there is a substantial difference between the real log-canonical threshold and the \textit{log-canonical threshold} which is an important invariant used in algebraic geometry (see e.g. \cite[Section 9.3.B]{lazarsfeld2004positivity}). Let $f\in\R[x_{1},\ldots, x_{d}]$ be a polynomial with real coefficients. By $f_{\C}$ we denote its \textit{complexification}, i.e. the same polynomial but as an element of $\C[x_{1},\ldots, x_{d}]$. Saito \cite{saito2007real} showed that ${\rm rlct}(f)\geq {\rm lct}(f_{\C})$. As an example let $f(x,y,z)=x^{2}+y^{2}+z^{2}$. By Koll\'{a}r \cite[Example 8.15]{kollar-sing-pairs}  we have ${\rm lct}_{0}(f_{\C})=1$ and we can easily show that over the real numbers a single blow-up at the origin (see e.g. \cite[Section 3.5]{watanabe_book}) allows us to compute the poles of $\zeta(z)$  (c.f. Proposition 3.3 in \cite{saito2007real}) giving ${\rm rlct}_{0}(f)=3/2$.
\end{rem}

\subsection{The marginal likelihood}\label{sec:marg-like}
Let $X$ be a discrete random variable with values in $[m]$ for some $m\geq 1$. A distribution of $X$ is given by $(p(X=1), \ldots, p(X=m))$. Denoting $p(X=i)$ by $p_{i}$ we associate each probability distribution for $X$ with a point $p=(p_{1},\ldots, p_{m})$ in the \textit{probability simplex}
$$
\Delta_{m-1}=\{x\in\R^{m}: x_{i}\geq 0, \sum_{i=1}^{m} x_{i}=1\}.
$$
By definition a model for $X$ is a family of points in $\Delta_{m-1}$. The model analysed in this paper is a special case of a parametric algebraic statistical models defined as an image in $\Delta_{m-1}$ of a polynomial mapping $p:\,\Theta\rightarrow \Delta_{m-1}$, where $\Theta\subseteq \R^{d}$ is called the parameter space (see e.g. Chapter 1, \cite{sturmfels}). We define $\cM=p(\Theta)$. Note that for a given integer $N$ every point $q\in\Delta_{m-1}$ gives a multinomial distribution ${\rm Mult}(N,q)$. Hence given a fixed $N$ we can naturally associate $\Delta_{m-1}$ with the multinomial model and hence $\cM$ can be treated as a submodel of the multinomial model. 

Let $X^{(N)}=(X^{1},\ldots, X^{N})$ denote $N$ independent observations of $X$ and let $(N_{i})$ for $i\in[m]$ be the sufficient statistic given by the sample counts. Let $\hat{p}^{(N)}=[\hat{p}_{i}^{(N)}]$ denote the sample proportions   $\hat{p}_{i}^{(N)}=N_{i}/N$. Given that the observations in $X^{(N)}$ are independent we can write the logarithm of the marginal likelihood $Z(N)$ as a function of $\hat{p}^{(N)}$. Let $\ell(p(\theta);X)=\log L(\theta;X)$ be the log-likelihood for a single observation. Then the observed log-likelihood of the data can be rewritten as 
\begin{equation}\label{eq:likelihood}
\ell_N(p(\theta))=\sum_{\a\in\{0,1\}^{n}} N_{\a} \log p_{\a}(\theta)= N \ell(p(\theta);\hat{p}^{(N)}).
\end{equation}
If the sample proportions $\hat{p}^{(N)}$ lie in the interior of the probability simplex then the likelihood function for the multinomial model as a function of the probabilities is always maximized over $\hat{p}^{(N)}$. Hence if $\hat{p}^{(N)}\in \cM$ the likelihood function constrained to $\cM$ is also maximized at $\hat{p}$. It follows that with assumption (A2) the maximum likelihood estimates for sufficiently large $N$ are given as all the points in the parameter space mapping to $\hat{p}$ which we denote by $\widehat{\Theta} =p^{-1}(\hat{p})\subset\Theta $.  


For given $\hat{p}$ define the normalized log-likelihood as a function $f:\Theta \rightarrow \R$ 
\begin{equation}\label{eq:def-f}
f(\theta)=f(p(\theta);\hat{p})=\ell(\hat{p};\hat{p})-\ell(p(\theta);\hat{p})\geq 0.
\end{equation}
Then $Z(N)$ in (\ref{eq:ZN0}) can be rewritten as $\exp({\hat{\ell}_{N}})\cdot I(N)$, where 
\begin{equation}\label{eq:ZN}
I(N)={\int_{\Theta } \exp\left\{ -N f(\theta))\right\}\varphi(\theta){\rm d}\theta}.
\end{equation}
The logarithm of the marginal likelihood can be written as $\log Z(N)=\hat{\ell}_{N}+\log I(N)$, where $\hat{\ell}_{N}=\ell_{N}(\hat{p})$. By construction $f\in\cA^{\geq}_{\Theta }$ and $f^{-1}(0)=p^{-1}(\hat{p})=\widehat{\Theta}$. By Theorem \ref{th:watanabe} to obtain the asymptotic approximation for $\log I(N)$, and hence also for $\log Z(N)$, we need to compute ${\rm RLCT}_{\Theta }(f;\varphi)$. 
\begin{rem}
Note that we are interested in the approximation of the observed marginal likelihood not in the expected one. Therefore, we cannot immediately apply the main result of Watanabe as stated in Theorem 1.1 in \cite{shaowei_rlct} (for a discussion see \cite{watanabe2010asymptotic}).
\end{rem}

In our analysis of general Markov models we distinguish two cases: the \textit{smooth case} when there exists a smooth manifold $M$ such that $\widehat{\Theta}=M\cap \Theta$ and the \textit{singular case}. The smooth case is simple to deal with. We can use the real log-canonical threshold to show that the BIC approximation in (\ref{eq:bicfinite}) generalizes to the case when $\widehat{\Theta}$ is a sufficiently regular subset of $\Theta$ given in (\ref{eq:bicsmotth}). We make this precise in the following proposition.
\begin{prop}\label{prop:regulargeneral}
Let $f\in\cA_{\Theta}^{\geq}(\R^{d})$ be the normalized log-likelihood in (\ref{eq:def-f}). Given (A1) and (A2) assume that  $\hat{p}=\hat{p}^{(N)}\in\cM$ is such that there exists a smooth manifold $M\subseteq \R^{d}$ satisfying $\widehat{\Theta}=M\cap \Theta$. Then as $N\rightarrow \infty$
$$
\log Z(N)=\hat{\ell}_{N}-\frac{d-d'}{2}\log N + O(1),
$$
where $d'=\dim \widehat{\Theta}$.
\end{prop}
\begin{proof}
By assumption (A1) there exist two constants $c,C>0$ such that $c<\varphi(\theta)<C$ on $\Theta$. Therefore 
$$
c\int_{\Theta}(f(\theta))^{-z}{\rm d}\theta<\zeta(z)<C\int_{\Theta}(f(\theta))^{-z}{\rm d}\theta
$$
and it follows that ${\rm RLCT}_{\Theta}(f;\varphi)={\rm RLCT}_{\Theta}(f)$. By Theorem \ref{th:watanabe} it suffices to prove the following lemma which generalises Proposition 3.3 in \cite{saito2007real}.
\begin{lem}\label{lem:smooth-rlct}
Let $\Theta\subset\R^{d}$ be a compact semianalytic set and $f\in\cA^{\geq}_{\Theta}(\R^{d})$. If there exists a smooth manifold $M\subseteq \R^{d}$ such that $\widehat{\Theta}=M\cap \Theta$ and $\theta_{0}\in \widehat{\Theta}$ then ${\rm RLCT}_{\theta_{0}}(f)={\rm RLCT}_{\Theta}(f)=(\frac{d-d'}{2},1)$ where $d'=\dim\widehat{\Theta}$. 
\end{lem}
To prove this recall that the real log-canonical threshold ${\rm RLCT}_{\theta_{0}}(f)$ does not depend on the choice of a neighborhood $W_{0}$ of $\theta_{0}$. Since $\widehat{\Theta}=M\cap \Theta$ and $M$ is a smooth manifold it follows that for each point $\theta_{0}$ of $\widehat{\Theta}$ there exists an open neighborhood $W_{0}$ of $\theta_{0}$ in $\R^{d}$ with local coordinates $w_{1},\ldots, w_{d}$ centered at $\theta_{0}$ such that the local equation of $X$ is $w_{1}^{2}+\cdots+w_{c}^{2}=0$, where $c=d-d'$. A single blow-up $\pi$ at the origin satisfies all the conditions of Hironaka's Theorem since in the new coordinates over one of the charts $f(\pi(u))=u_{1}^{2}a(u)$ where $a(u)$ is nowhere vanishing and $\pi'(u)=u_{1}^{c-1}$. For other charts the situation is the same and hence ${\rm RLCT}_{\theta_{0}}(f)=(c/2,1)$. Since by (\ref{eq:rlct-min}) ${\rm RLCT}_{\Theta}(f)=\min_{\theta_{0}\in\Theta}{\rm RLCT}_{W_{0}\cap \Theta}(f)$ it suffices to show that if $\theta_{0}$ is a boundary point of $\Theta$ then ${\rm RLCT}_{W_{0}\cap\Theta}(f)\geq (c/2,1)$. But this follows from (\ref{eq:rlct-constr}) and the fact that ${\rm RLCT}_{\theta_{0}}(f)= (c/2,1)$ as $\theta_{0}$ is a smooth point of $M$. The lemma is hence proved. 
\end{proof}

\section{General Markov models}\label{sec:tree-models}

In this section we formally define the general Markov model and give the asymptotic approximation for the marginal likelihood in the smooth case which is given by Theorem \ref{prop:regularcase}.

\subsection{Definition of the model class}\label{sec:model} All random variables considered in this paper are assumed to be binary with values in either $0$ or $1$. Let $T^{r}=(V,E)$ be a rooted tree. Recall that $n_{e}=|E|$ and $n_{v}=|V|$. For any $e=(k,l)\in E$ we say that $k$ and $l$ are \textit{adjacent} and $k$ is a \textit{parent} of $l$ and we denote it by $k=\pa(l)$. For every $\beta\in\{0,1\}^{V}$ let $p_{\beta}=\P(\bigcap_{v\in V}\{Y_{v}=\beta_{v}\})$. A \textit{Markov process} on a rooted tree $T^{r}$ is a sequence $Y=(Y_v)_{v\in V}$ of random variables such that for each $\beta=(\beta_{v})_{v\in V}\in \{0,1\}^{{V}}$
\begin{equation}\label{eq:p_albar}
p_\beta(\theta)=\theta^{(r)}_{\beta_{r}}\prod_{v\in V\setminus r} \theta^{(v)}_{\beta_{v}|\beta_{\pa(v)}},
\end{equation}
where $\theta^{(v)}_{\beta_{v}|\beta_{\pa(v)}}=\P(Y_v={\beta}_v|Y_{\pa(v)}=\beta_{\pa(v)})$ and $\theta^{(r)}_{\beta_{r}}=\P(Y_{r}=\beta_{r})$. In a more standard statistical language these models are just fully observed Bayesian networks on rooted trees. Since $\theta^{(v)}_{0|i}+\theta^{(v)}_{1|i}=1$ for all $v\in V$ and $i=0,1$ then the Markov process on $T^{r}$ defined by Equation (\ref{eq:p_albar}) has exactly $2{n_{e}}+1$ free parameters in the vector $\theta$:  one for the root distribution $\theta^{(r)}_{1}$ and two for each edge $(u,v)\in E$ given by  $\theta^{(v)}_{1|0}$ and $\theta^{(v)}_{1|1}$ and the vector of all parameters is denoted by $\theta$. The parameter space is $\Theta_T=[0,1]^{2{n_{e}}+1}$. Henceforth we usually omit the root $r$ in the notation writing $T$ to denote the rooted tree $T^{r}$.

The general Markov model on $T$ is induced from the Markov process  on $T$ by assuming that all the inner nodes represent hidden random variables. Hence we consider induced marginal probability distributions over the leaves of $T$. The set of leaves is denoted by $L$. We assume that $T$ has $n$ leaves and hence we can associate $L$ with $[n]$ with some arbitrary numbering of the leaves.  Let $Y=(X,H)$ where $X=(X_1,\ldots, X_n)$ denotes the variables represented by the leaves of $T$ and $H$ denotes the vector of variables represented by inner nodes, i.e. $X=(Y_{v})_{v\in L}$ and $H=(Y_{v})_{v\in V\setminus L}$. We define the general Markov model $\cM_{T}$ to be the model in the probability simplex $\Delta_{2^{n}-1}$ obtained by summing out in (\ref{eq:p_albar}) all possible values of the inner nodes. By definition $\cM_{T}$ is the image of the map $p:\Theta_T\rightarrow \Delta_{2^n-1}$ given by
\begin{equation}\label{eq:p_albar2}
p_{\a}(\theta)=\sum_{\cH} \theta^{(r)}_{\beta_{r}}\prod_{v\in V\setminus r} \theta^{(v)}_{\beta_{v}|\beta_{\pa(v)}}\quad \mbox{ for any $\a\in\{0,1\}^{L}$},
\end{equation}
where $\cH$ is the set of all vectors $\beta=(\beta_{v})_{v\in V}$ such that $(\beta_{v})_{v\in L}=\a$. For a more detailed treatment see Chapter 8 in \cite{semple2003pol}.


\subsection{The smooth case}\label{sec:smooth}

For $\hat{p}\in\cM_T$ let $\widehat{\Sigma}=[\hat{\mu}_{ij}]\in\R^{n\times n}$ be the covariance matrix of the random vector represented by the leaves of $T$. In \cite{pwz2010-identifiability} we show that the geometry of the $\hat{p}$-fiber $\widehat{\Theta}_{T}$ is determined by zeros in $ \widehat{\Sigma}$. We say that that an edge $e\in E$ is \textit{isolated relative to} $\hat{p}$ if $\hat{\mu}_{ij}=0$ for all $i,j\in[n]$ such that $e\in E(ij)$, where $E(ij)$ denotes the set of edges in the path joining $i$ and $j$. By $\widehat{E}\subseteq E$ we denote the set of all  edges of $T$ which are isolated relative to $\hat{p}$. By $\widehat{T}=(V,E\setminus \widehat{E})$ we denote the forest obtained from $T$ by removing edges in $\widehat{E}$. 
 
We now define relations on $\widehat{E}$ and $E\setminus \widehat{E}$. For two edges $e,e'$ with either $\{e,e'\}\subset \widehat{E}$ or $\{e,e'\}\subset E\setminus\widehat{E}$ write $e\sim e'$ if either $e=e'$ or $e$ and $e'$ are adjacent and all the edges that are incident with both $e$ and $e'$ are isolated relative to $\hat{p}$. Let us now take the transitive closure of $\sim$ restricted to pairs of edges in $\widehat{E}$ to form an equivalence relation on $\widehat{E}$. Similarly, take the transitive closure of $\sim$ restricted to the pairs of edges in $E\setminus \widehat{E}$ to form an equivalence relation in $E\setminus \widehat{E}$. We will let $[\widehat{E}]$ and $[E\setminus \widehat{E}]$ denote the set of equivalence classes of $\widehat{E}$ and $E\setminus \widehat{E}$ respectively. 

By the construction all the inner nodes of $T$ have either degree zero in $\widehat{T}$ or the degree is strictly greater than one. We say that a node $v\in V$ is \textit{non-degenerate with respect to} $\hat{p}$ if either $v$ is a leaf of $T$ or $\deg v\geq 2$ in $\widehat{T}$. Otherwise we say that the node is \textit{degenerate with respect to} $\hat{p}$. Note that this coincides with the definition of a degenerate node given in the introduction. The set of all nodes which are degenerate with respect to $\hat{p}$ is denoted by $\widehat{V}$. 

\begin{prop}[\cite{pwz2010-identifiability}, Theorem 5.4]\label{prop:regular}
Let $T$ be a tree with $n$ leaves. Let $\hat{p}\in\cM_{T}$ and let $\widehat{T}$ be defined as above. If each of the inner nodes of $T$ has degree at least two in $\widehat{T}$ then  $\widehat{\Theta}_{T}$ is a manifold with corners and $\dim\widehat{\Theta}_{T}=2l_{2}$, where $l_{2}$ is the number of nodes which have degree two in $\widehat{T}$. 
\end{prop}

For the case covered by Proposition \ref{prop:regular} we obtain a way to compute the asymptotic approximation for the marginal likelihood. 


\begin{prop}\label{prop:reg}
Let $\hat{p}\in\cM_{T}$ be such that each inner node of $T$ has degree at least two in $\widehat{T}$ and let $f$ be the normalized likelihood defined by (\ref{eq:def-f}). Then 
$$
{\rm RLCT}_{\Theta_{T}}(f)=\left(\frac{n_{v}+n_{e}-2l_{2}}{2},1\right).
$$
\end{prop}
\begin{proof}
Since every inner node of $T$ has degree at least two in $\widehat{T}$ then by Proposition \ref{prop:regular} there exists a smooth manifold $M\subseteq\R^{n_{v}+n_{e}}$ such that $\widehat{\Theta}_{T}=M\cap\Theta_{T}$ and $\dim\widehat{\Theta}=2 l_{2}$. The result follows from Proposition \ref{prop:regulargeneral} and the fact that $\dim \Theta_{T}=n_{v}+n_{e}$. \end{proof}

By Theorem \ref{th:watanabe}, Proposition \ref{prop:reg} implies Theorem \ref{prop:regularcase} since $l_{2}$ in its statement is exactly the number of inner nodes $v$ such that the degree of $v$ in $\widehat{T}$ is two. 

\begin{rem} Theorem \ref{prop:regularcase} is still true if (A1) is replaced by the assumption that the prior distribution is bounded on $\Theta_{T}$ and there exists an open subset of $\Theta_{T}$ with a non-empty intersection with $\widehat{\Theta}_{T}$ where the prior is strictly positive. In particular we can use conjugate Beta priors  $\theta^{(v)}_{1|i}\sim {\rm Beta}(\a^{(v)}_{i},\beta^{(v)}_{i})$ as long as $\a^{(v)}_{i},\beta^{(v)}_{i}\geq 1$.\end{rem}

\section{The ideal-theoretic approach}\label{sec:sing-case}

In this section we define the real log-canonical threshold of an ideal. Theorem \ref{th:barI} translated the problem of finding the real log-canonical threshold of the normalized log-likelihood into algebra. We then analyse the case of general Markov models. In Theorem \ref{th:redtokappas} we apply a useful change of coordinates which enables us to deal with the singular case in a more efficient. 

\subsection{The real log-canonical threshold of an ideal}
Let $f_{1},\ldots, f_{r}\in\cA_{\Theta}$ then the \textit{ideal generated} by $f_{1},\ldots,f_{r}$ is denoted by $$\la f_{1},\ldots,f_{r}\ra=\{f\in\cA_{\Theta}:\,\,f(\theta)=\sum_{i=1}^{r} h_{i}(\theta)f_{i}(\theta), h_{i}\in \cA_{\Theta}\}.$$ 
Following \cite{shaowei_rlct} we generalize the notion of the real log-canonical thresholds to the ideal $ I=\la f_{1},\ldots,f_{r}\ra$. This mirrors the analytic definition of the log-canonical threshold of an ideal (see e.g. \cite[Section 9.3.D]{lazarsfeld2004positivity}). By definition 
\begin{equation}\label{eq:idealRLCT}
{\rm RLCT}_{\Theta}( I;\varphi)={\rm RLCT}_{\Theta}(\la f_{1},\ldots, f_{r}\ra ;\varphi):={\rm RLCT}_{\Theta}(f;\varphi),
\end{equation}
 where $f(\theta)=f_{1}^{2}(\theta)+\cdots+f_{r}^{2}(\theta)$. 
By Proposition 4.5 in \cite{shaowei_rlct} the real log-canonical threshold does not depend on the choice of generators. We say that $I=\la f_{1},\ldots, f_{r}\ra$ is $\R$-nondegenerate if  $f=\sum f_{i}^{2}$ is $\R$-nondegenerate as given by Definition \ref{def:nondegenerate}. 

The following important proposition enables us to use the full power of the ideal-theoretic approach.
\begin{prop}\label{lem:rlct-repar}\label{lem:red-to-1}\label{lem:szacowanie}
Let $f,g\in\cA_{\Theta}(\R^{d})$ and let $ I$  be an ideal in $\cA_{\Theta}(\R^{d})$. Then 
\begin{description}
\item[i] Let $\rho:\Omega\rightarrow\Theta$ be a proper real analytic isomorphism. Denote  $\rho^{*} I=\{f\circ\rho:\,f\in I\}$ be the pullback of $ I$ on $\cA_{\Omega}$. Then, 
$$
{\rm RLCT}_{\Theta}( I;\varphi)={\rm RLCT}_{\Omega}(\rho^{*} I;(\varphi\circ \rho)|\rho'|),
$$ 
where $|\rho'|$ denotes the Jacobian of $\rho$.
\item[ii] If $\varphi$ is positive and bounded on $\Theta$ then
$$
{\rm RLCT}_{\Theta}( I;\varphi)={\rm RLCT}_{\Theta}( I).
$$
\item[iii] If there exist constants $c,c'>0$ such that $c\, g(\theta)\leq f(\theta)\leq c' g(\theta)$ for every $\theta\in \Theta$ then ${\rm RLCT}_{\Theta}(f)={\rm RLCT}_{\Theta}(g)$.
\item[iv] Let $ I=\la f_{1},\ldots, f_{r}\ra$ and $J=\la g_{1},\ldots, g_{r}\ra$ where $g_{i}=u_{i}f_{i}$ for $i=1,\ldots,r$ and there exist positive constants $c,C$ such that $c<u_{i}(\theta)<C$ for all $\theta\in \Theta$ and for all $i=1,\ldots,r$. Then ${\rm RLCT}_{\Theta}(I)={\rm RLCT}_{\Theta}(J)$.
\end{description}
\end{prop}
\begin{proof}
The first result is a direct consequence of Proposition 4.7, \cite{shaowei_rlct}. The second and the third follow easily from the interpretation of the real log-canonical threshold and its multiplicity as coefficients in the asymptotic approximation of $I(N)$ in (\ref{eq:laplace2}). The last statement follows from the third.
\end{proof}
In the statistical context given in Section \ref{sec:marg-like}, expressing this problem in the language of ideals simplifies reductions. 
\begin{thm}\label{th:barI} Let $p=(p_{1},\ldots,p_{m}):\, \Theta\rightarrow \Delta_{m-1}$ be a polynomial mapping and $\cM=p(\Theta)$ be the statistical model of $X\in [m]$. For a given $\hat{p}\in\cM $ define
\begin{equation}\label{eq:idealI0}
\overline{\cI} =\la p_{1}(\theta)-\hat{p}_{1},\ldots, p_{m}(\theta)-\hat{p}_{m}\ra\subset\cA_{\Theta }.\end{equation}
Let $f$ denote the normalized likelihood defined by (\ref{eq:def-f}) and $\varphi$ the prior distribution on $\Theta$ satisfying (A1). Then we have that
\begin{equation}\label{eq:thm1.2}
{\rm RLCT}_{\Theta }(f;\varphi)={\rm RLCT}_{\Theta }(\overline{\cI};\varphi)={\rm RLCT}_{\Theta }(\overline{\cI}).
\end{equation}
\end{thm}
\begin{proof}
The first equation follows essentially the same lines as the proof of Theorem 1.2 in \cite{shaowei_rlct}. The second equation follows from Proposition \ref{lem:szacowanie} ii. \end{proof}

\subsection{A reparameterization of the model}

To realize how the formulation in terms of the ideals may be useful we first need to introduce a new generating set for the ideal $\overline{\cI}$ in (\ref{eq:idealI0}). By Proposition 4.5 in \cite{shaowei_rlct} this does not affect our computations. Then we take a pullback of $\overline{\cI}$ under a polynomial isomorphism. We note that this is the algebraized version of the analytic reductions applied in \cite{rusakov2006ams}.

Following \cite{pwz2010-identifiability} we perform a change of coordinates on the model space and parameter space. Let $T$ be a tree with $n$ leaves. In this case the change of the generating set for $\overline{\cI}$ is induced by the following series of transformations. First we express the raw probabilities $p_{\a}$ for $\a\in\{0,1\}^{n}$ in terms of a new system of variables given by the non-central moments $\lambda_{\a}=\E X^{\a}$. This change is a simple linear map $f_{p\lambda}:\R^{2^{n}}\rightarrow\R^{2^{n}}$ with the determinant equal to one. Thus $\mathbf{\lambda}=f_{p\lambda}(p)$ is defined as follows
\begin{equation}\label{eq:lambda_in_p}
\lambda_\a=\sum_{\a\leq \beta\leq \mathbf{1}} p_\beta\qquad\mbox{ for any } \a\in\{0,1\}^{n}, 
\end{equation}
where $\mathbf{1}$ denotes here the vector of ones and the sum is over all binary vectors $\beta$ such that $\a\leq \beta\leq \mathbf{1}$ in the sense that $\a_i\leq \beta_i\leq 1$ for all $i=1,\ldots,n$. To define the next change of coordinates we change indexing of the non-central moments in such a way that $\lambda_{\a}$ for $\a\in\{0,1\}^{n}$ is denoted by $\lambda_{A}$ for $A\subseteq [n]$ where $i\in A$ if and only if $\a_{i}=1$. The linearity of the expectation implies that the central moments can be expressed in terms of the non-central moments. We have  
\begin{equation}\label{eq:central}
\mu_I=\sum_{J\subseteq I}(-1)^{|J|} \lambda_{I\setminus J} \prod_{i\in J} \lambda_{i} \quad \mbox{ for }I\subseteq [n], \,\,|I|\geq 2. 
\end{equation}
Moreover, there is an algebraic isomorphism between the non-central moments $\lambda_{I}$ for all non-empty $I\subseteq [n]$ and all the means $\lambda_{i}=\E X_{i}$ supplemented with all the central moments $\mu_{I}=\E \left(\prod_{i\in I} (X_{i}-\lambda_{i})\right)$ for $I\in [n]_{\geq 2}$, where $[n]_{\geq k}$ denotes all subsets of $[n]$ with at least $k$ elements (see Appendix A in \cite{pwz2010-identifiability}). In particular we obtain in this way a change of coordinates from the non-central moments $\lambda_{I}$ for $I\subseteq [n]$ to the central moments supplemented by the means.

To specify the last change of coordinates we need some basic combinatorics. We define a partially ordered set (or poset) $\Pi_{T}$  of all partitions of the set of leaves $[n]$ obtained by removing some inner edges of $T$ and considering connected components of the resulting forest. The elements of $\Pi_{T}$ are partitions $\pi=B_{1}|\cdots|B_{k}$ where each $B_{i}$ is called a \textit{block} of $\pi$. The ordering on this poset is induced from the ordering on the poset of all partitions on $[n]$ (see Example 3.1.1.d, \cite{stanley2006enumerative}). Thus for $\pi=B_{1}|\cdots|B_{k}$ and $\nu=C_{1}|\cdots|C_{l}$ we write $\pi\leq \nu$ if and only if every block of $\pi$ is contained in one of the blocks of $\nu$.

 For any poset $\Pi$ we define its M\"{o}bius function $\mm:\Pi\times \Pi\rightarrow\R$ (c.f. \cite[Chapter 3]{stanley2006enumerative}) by setting
$$\mm(\pi,\pi)=1 \mbox{ for every }\pi\in\Pi\,\mbox{ and }\,\mm(\pi,\nu)=-\sum_{\pi\leq \delta\leq \nu}\mm_{I}(\pi,\delta).$$  For any $I\subseteq [n]$ we define $T(I)$ as the minimal subtree of $T$ containing $I$ in its vertex set. By $\mm_{I}$ we denote the M\"{o}bius function of $\Pi_{T(I)}$ and by $\hat{1}_{I}$ the maximal one-block partition of $\Pi_{T(I)}$. 

The last system of coordinates is given by $n$ means $\lambda_{i}$ for $i\in 1,\ldots,n$ and $\kappa_{I}$ for all $I\in[n]_{\geq 2}$, where 
\begin{equation}\label{eq:kappa}
\kappa_I=\sum_{\pi\in\Pi_{{T}(I)}} \mm_I(\pi,\hat{1}_{I}) \prod_{B\in \pi}\mu_{B}.
\end{equation}
We note that in particular $\kappa_{ij}=\mu_{ij}$ for all $i,j\in [n]$. This change of coordinates, from probabilities $p_{\a}$ for $\a\in\{0,1\}^{n}$ to $(\lambda_{i},\kappa_{I})$ for $i=1,\ldots,n$ and $I\in [n]_{\geq 2}$ is denoted by $f_{p\kappa}$. It is an algebraic isomorphism given that $\sum_{\a}p_{\a}=1$ (see Appendix A, \cite{pwz2010-identifiability}). The new set of coordinates is called the system of \textit{tree cumulants}.
\begin{exmp}\label{ex:quartet1}
Consider the quartet tree model, i.e. the hidden tree Markov model given by the graph in Figure \ref{fig:quartet}.
\begin{figure}[h]
\centering
\includegraphics{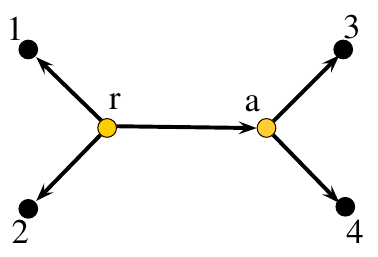}
\caption{A quartet tree}\label{fig:quartet}
\end{figure}
The tree cumulants are given by $15$ coordinates: $\lambda_{i}=\E  X_{i}$ for $i=1,2,3,4$ and  $\kappa_{I}$ for $I\in[4]_{\geq 2}$. We have $\kappa_{ij}=\mu_{ij}={\rm Cov}(X_{i},X_{j})$ for $1\leq i<j\leq 4$ and $\kappa_{ijk}=\mu_{ijk}$ for all $1\leq i<j<k\leq 4$. However tree cumulants of higher order cannot be equated to corresponding central moments but only expressed as functions of them. Thus in this case by (\ref{eq:kappa})
$$
\kappa_{1234}=\mu_{1234}-\mu_{12}\mu_{34}.
$$ 
\end{exmp}

The next step is to change the coordinates on the parameter space of the model. Define the following set of ${n_{v}}+{n_{e}}$ parameters. For every directed edge $(u,v)\in E$ let
\begin{eqnarray}\label{eq:uij}
\eta_{u,v}=\theta^{(v)}_{1|1}-\theta^{(v)}_{1|0} \mbox{ and}\\
\nonumber s_v=1-2\lambda_v\quad\mbox{for each } v\in V,
\end{eqnarray}
where $\lambda_{v}=\E Y_{v}$ is a polynomial in the original parameters $\theta$. Let $(r,v_{1},\ldots, v_{k},v)$ be a directed path in $T$. Then
$$
\lambda_{v}=\sum_{\a\in\{0,1\}^{k+1}} \theta^{(v)}_{1|\a_{k}}\theta^{(v_{k})}_{\a_{k}|\a_{k-1}}\cdots \theta^{(r)}_{\a_{r}}.
$$
We denote the new parameter space by $\Omega_T$ and the coordinates by  $\omega=((s_v),(\eta_{u,v}))$ for ${v\in V, (u,v)\in E}$. 

Simple linear constraints defining $\Theta_{T}$ become only slightly more complicated when expressed in the new parameters. The choice of parameter values is not free anymore in the sense that constraints for each of the parameters involve other parameters. $\Omega_{T}$ is given by $s_{r}\in [-1,1]$ and for each $(u,v)\in E$ (c.f. Equation (19) in \cite{pwz2010-identifiability})
\begin{equation}\label{eq:constraints}
\begin{array}{l}
-(1+s_{v}) \leq (1-s_{u})\eta_{u,v} \leq (1-s_{v})\\
-(1-s_{v}) \leq (1+s_{u})\eta_{u,v} \leq (1+s_{v}).
\end{array}
\end{equation}
Since $T$ is a tree then  $2{n_{e}}+1={n_{v}}+{n_{e}}$ and hence $\dim \Omega_T=\dim \Theta_T$. The change of parameters defined above is denoted by $f_{\theta\omega}:\,\Theta_T\rightarrow \Omega_T$. It is a polynomial isomorphism with the inverse denoted by $f_{\omega\theta}$ (see \cite[Section 4]{pwz2010-identifiability}). Recall that for $I\subseteq [n]$ by $T(I)=(V(I),E(I))$ we denote the subtree of $T$ spanned on $I$. Let $r(I)$ denote the root of $T(I)$. Then for instance if $T$ is the quartet tree in Figure \ref{fig:quartet} then for $T(34)$:  $E(34)=\{(a,3),(a,4)\}$ and $r(34)=a$. The parameterization of $\cM_T$ in the system of tree cumulants is given as a map $\psi_{T}:\Omega_{T}\rightarrow f_{p\kappa}(\Delta_{2^{n}-1})$ by the following proposition.
\begin{prop}[Proposition 4.1 in \cite{pwz2010-identifiability}]\label{lem:param}
Let $T=(V,E)$ be a rooted trivalent tree with $n$ leaves. Then for each $i=1,\ldots, n$ one has $\lambda_{i}=\frac{1}{2}(1-s_{i})$ and
\begin{equation}\label{eq:param}
\kappa_I(\omega)=\frac{1}{4}(1-s_{r(I)}^2)\prod_{v\in V(I)\setminus I} s_v^{\deg v-2}  \prod_{(u,v)\in E(I)} \eta_{u,v}\quad\mbox{ for all } I\in [n]_{\geq 2},
\end{equation}
where the degree of $v\in V(I)$ is considered in the subtree $T(I)$. 
\end{prop}

We obtain the following diagram where the induced parameterisation $\psi_{T}$ is given in the bottom row.
\begin{equation}\label{eq:diagram}
\xymatrixcolsep{4pc}\xymatrix{
\Theta_{T} \ar@<1ex>[d]^{f_{\theta\omega}} \ar[r]^{f_{T}} &\Delta_{2^{n}-1}\ar@<1ex>[d]^{f_{p\kappa}}\\
\Omega_{T}\ar@<1ex>[u]^{f_{\omega\theta}} \ar@{-->}[r]^{\psi_{T}} &\cK_{T}\ar@<1ex>[u]^{f_{\kappa p}}}
\end{equation}
Let $\cI$ denote the pullback of the ideal $\overline{\cI}\subseteq \cA_{\Theta_{T}}$ to the ideal in $\cA_{\Omega_{T}}$induced by $f_{\theta\omega}$. Thus $\cI=f_{\omega\theta}^{*}\overline{\cI}=\{f\circ f_{\omega\theta}:\, f\in \overline{\cI}\}$. The ideal describes $\widehat{\Omega}_{T}=f_{\theta\omega}(\widehat{\Theta}_{T})$ as a subset of $\Omega_{T}$. The pullback of $\overline{\cI}$ satisfies
\begin{equation}\label{eq:pullback}
{\cI}=\la \lambda_{1}-\hat{\lambda}_{1},\ldots, \lambda_{n}-\hat{\lambda}_{n}\ra+\left(\sum_{I\in[n]_{\geq 2}}\la \kappa_{I}(\omega)-\hat{\kappa}_{I}\ra\right),
\end{equation}
where $\hat{\lambda}_{i}$ and $\hat{\kappa}_{I}$ are the corresponding coordinates of $f_{p\kappa}(\hat{p})$. Here the sum of ideals results in another ideal with the generating set which is the sum of generating sets of the summands.

For local computations we use the following reduction.
\begin{prop}[Proposition 4.6 in \cite{shaowei_rlct}]\label{lem:product-rlct}
Let $I\subseteq \cA_{x_{0}}(\R^{m})$, $ J\subseteq\cA_{y_{0}}(\R^{n})$ be two ideals. If  ${\rm RLCT}_{x_{0}}( I)=(\lambda_{x},m_{x})$ and ${\rm RLCT}_{y_{0}}( J)=(\lambda_{y},m_{y})$ then $${\rm RLCT}_{(x_{0},y_{0})}( I+ J)=(\lambda_{x}+\lambda_{y},m_{x}+m_{y}-1).$$
\end{prop}

\begin{thm}\label{th:redtokappas}
Let $T$ be a rooted tree with $n$ leaves and $\hat{p}\in \cM_{T}$.  Let $\overline{\cI}$ be the ideal defined by (\ref{eq:idealI0}) and $\cI$ the ideal defined by (\ref{eq:pullback}). Then
\begin{equation}\label{eq:min-w-omega}
{\rm RLCT}_{\Theta_{T}}(\overline{\cI})={\rm RLCT}_{\Omega_{T}}(\cI)=\min_{\omega_{0}\in\widehat{\Omega}_{T}} {\rm RLCT}_{\Omega_{0}}(\cI),
\end{equation}
where $\Omega_{0}$ is a sufficiently small neighborhood of $\omega_{0}$ in $\Omega_{T}$. Let $\cJ=\sum_{I\in[n]_{\geq 2}}\la \kappa_{I}(\omega)-\hat{\kappa}_{I}\ra$ then for every $\omega_{0}\in \widehat{\Omega}_{T}$
\begin{equation}\label{eq:Inadwa}
{\rm RLCT}_{\omega_{0}}(\cI)=\left(\frac{n}{2},0\right)+{\rm RLCT}_{\omega_{0}}(\cJ).
\end{equation}
\end{thm}
\begin{proof}
Since $f_{\omega\theta}$ is an isomorphism with a constant Jacobian then the first part of the theorem follows from Proposition \ref{lem:rlct-repar} (i).

Let $W$ be an $\e$-box around $\omega_{0}$. If $T$ is rooted in an inner leaf then by Proposition \ref{lem:param} the ideal $\cJ$ does not depend on $s_{1},\ldots, s_{n}$. Since for every $i=1,\ldots, n$ the expression $\lambda_{i}-\hat{\lambda}_{i}$ depends only on $s_{i}$ then 
$${\rm RLCT}_{\omega_{0}}(\la \lambda_{1}-\hat{\lambda}_{1},\ldots, \lambda_{n}-\hat{\lambda}_{n}\ra)=(\frac{n}{2},1),$$
which can be easily checked (see e.g. Proposition 3.3 in \cite{saito2007real}). Equation (\ref{eq:Inadwa}) follows from Proposition \ref{lem:product-rlct}. 

Now assume that $T$ is rooted in one of the leaves. In this case both $\la \lambda_{1}-\hat{\lambda}_{1},\ldots, \lambda_{n}-\hat{\lambda}_{n}\ra$ and $\cJ$ depend on $s_{r}$ because $\kappa_{I}(\omega)=(1-s_{r}^{2})f_{I}(\omega)$ for some monomial $f_{I}(\omega)$ whenever $r\in I$. Therefore we cannot use Proposition \ref{lem:product-rlct} directly. However, by assumption (A2) $s_{i}^{0}\in (-1,1)$ for $i=1,\ldots, n$. Hence for each $\omega_{0}$ and a sufficiently small $\e$ one can find two  positive constants $c(\e), C(\e)$ such that $c(\e)\leq 1-s_{r}^{2}\leq C(\e)$ in $W$. By Proposition \ref{lem:szacowanie} (iv) the real log canonical threshold of $\cJ$ in $W$ is equal to the real log-canonical threshold of a an ideal with generators induced from the generators of $\cJ$ by replacing each $ 1-s_{r}^{2}$ by $1$. Now again (\ref{eq:Inadwa}) follows from Proposition \ref{lem:product-rlct}.
\end{proof}

\section{The main reduction step}\label{sec:splittrees}

In this section we show that the computations can be reduced to two main cases. First, when $\hat{p}$ is such that $\hat{\kappa}_{ij}\neq 0$ for all $i,j\in [n]$. Second, when $\hat{p}$ is such that $\hat{\kappa}_{ij}= 0$ for all $i,j\in [n]$. Moreover, the second case is reduced to computations for monomial ideals which are usually amenable to various combinatorial techniques.

Let $T$ be a trivalent tree with $n\geq 3$ leaves and let $\hat{p}\in \cM_{T}$. If all the equivalence classes in $[\widehat{E}]$ are singletones or $[\widehat{E}]$ is empty, which is equivalent to the fact that every inner node has degree at least two in $\widehat{T}$, then Theorem \ref{prop:regularcase} gives us the asymptotic approximation for the marginal likelihood. Thus let assume that there is at least one nontrivial class in $[\widehat{E}]$. Let $T_1,\ldots,T_k$ denote trees representing the equivalence classes in  $[\widehat{E}]$ and let $S_1,\ldots,S_m$ denote trees induced by the connected components of  $E\setminus\widehat{E}$. Let $L_1,\ldots L_k$  denote the sets of leaves of $T_1,\ldots, T_k$. For each $S_{i}$ $i=1,\ldots, m$ by Remark 5.2 (iv) in \cite{pwz2010-identifiability} its set of leaves denoted by $[n_{i}]$ is a subset of $[n]$. We illustrate this notation  using the graph below where the dashed edges represent edges in $\widehat{E}$.
\begin{center}
\includegraphics[scale=0.7]{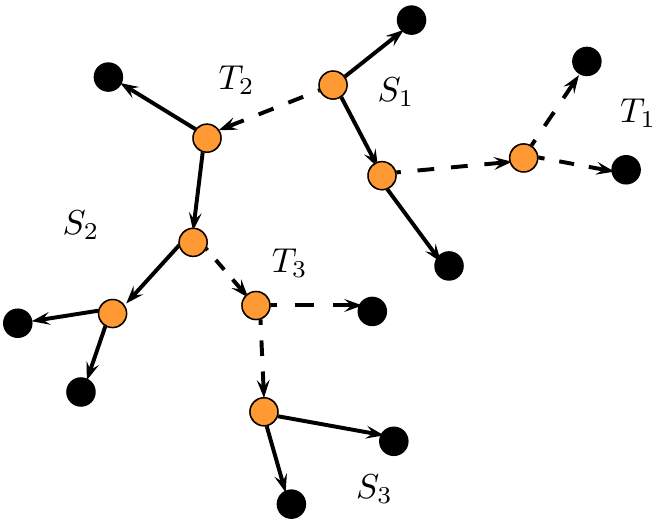}
\end{center}

\begin{lem}\label{lem:redzQ}
Let $T=(V,E)$ be a trivalent rooted tree with $n\geq 4$ leaves and let $\hat{p}\in\cM_{T}$. Let $\cJ=\sum_{I\in[n]_{\geq 2}}\la\kappa_{I}(\omega)-\hat{\kappa}_{I}\ra$. If $\omega_{0}\in \widehat{\Omega}_{T}$ then 
\begin{equation}\label{eq:IisJplussth}
{\rm RLCT}_{\omega_{{0}}}(\cJ)=\sum_{i=1}^{m}{\rm RLCT}_{\omega_{{0}}}(\cJ(S_{i}))+\sum_{i=1}^{k}{\rm RLCT}_{\omega_{{0}}}(\cJ(T_{i}))+(0,1-m-k),
\end{equation}
where $\cJ(S_{i})=\sum_{I\in[n_{i}]_{\geq 2}}\la\kappa_{I}(\omega)-\hat{\kappa}_{I}\ra$ for $i=1,\ldots,m$ and $\cJ(T_{i})=\sum_{w,w'\in L_{i}}\la\kappa_{ww'}(\omega)\ra$ for $i=1,\ldots, k$.
\end{lem}
\begin{proof}
We first show that $\sum_{I:\hat{\kappa}_{I}=0}\la\kappa_{I}(\omega)\ra=\sum_{i,j:\hat{\kappa}_{ij}=0}\la\kappa_{ij}(\omega)\ra$. The inclusion ``$\supseteq$'' is clear. We now show ``$\subseteq$''. First note that for every $I\in [n]_{\geq 2}$ if $\hat{\kappa}_{I}=0$ then either $\eta_{e}^{0}=0$ for an edge $e\in E(I)$ or $s_{r(I)}^{2}=1$. It is easy to check there exist $i,j\in I$ such that such that $\hat{\kappa}_{ij}=0$ and the $r(ij)=r(I)$. It follows by Proposition \ref{lem:param} that $\kappa_{I}(\omega)=\kappa_{ij}(\omega) f(\omega)$ for a polynomial $f(\omega)$ and therefore the inclusion ``$\subseteq$'' is also true. This implies
$$\cJ=\sum_{I:\hat{\kappa}_{I}\neq 0}\la\kappa_{I}(\omega)-\hat{\kappa}_{I}\ra+\sum_{I:\hat{\kappa}_{I}= 0}\la\kappa_{I}(\omega)\ra= \sum_{i=1}^{m}\cJ(S_{i}) +\sum_{i,j:\hat{\kappa}_{ij}= 0}\la\kappa_{ij}(\omega)\ra.$$ 
Hence to proof the lemma it suffices to show that for every $\omega_{0}\in \widehat{\Omega}_{T}$
\begin{equation}\label{eq:rlctredstep}
{\rm RLCT}_{\omega_{0}}(\sum_{i=1}^{m}\cJ(S_{i})+ \sum_{i,j:\hat{\kappa}_{ij}=0}\la\kappa_{ij}(\omega)\ra) 
\end{equation}
is equal to the right hand side of (\ref{eq:IisJplussth}). 

If $e\in E\setminus \widehat{E}$ then by definition there exist $i,j\in [n]$ such that $\hat{\kappa}_{ij}\neq 0$ and $e\in E(ij)$. Since by Proposition \ref{lem:param} $\hat{\kappa}_{ij}=\eta_{e}^{0}f(\omega_{0})$ for a polynomial $f$ then in particular $\eta_{e}^{0}\neq 0$. It follows that for a sufficiently small $\e$ for each $E'\subseteq E\setminus \widehat{E}$ one can find positive constants $c(\e),C(\e)$ such that $c(\e)\leq \prod_{e\in E'}\eta_{e}\leq C(\e)$ holds in the $\e$-box around $\omega_{0}$. Similarly if $v\notin \widehat{V}$ (c.f. Section \ref{sec:smooth}) then there exist positive constants $d(\e), D(\e)$ such that $d(\e)\leq (1-s_{v}^{2})\leq D(\e)$ in the $\e$-box around $\omega_{0}$. It follows by Proposition \ref{lem:szacowanie} (iv) that in computations of the real log-canonical threshold in (\ref{eq:rlctredstep}) each $\kappa_{ij}(\omega)$ can be replaced by
\begin{equation}\label{eq:monmom1}
(1-s_{r(ij)}^{2})^{\delta_{r(ij)}}\!\!\!\!\!\!\prod_{e\in E(ij)\cap \widehat{E}} \!\!\!\eta_{e}
\end{equation}
where $\delta_{r(ij)}=1$ if  $r(ij)\in \widehat{V}$ and $\delta_{r(ij)}=0$ otherwise. Thus in (\ref{eq:rlctredstep}) we can replace the ideal $\sum_{i,j:\hat{\kappa}_{ij}=0}\la\kappa_{ij}(\omega)\ra$ by the ideal $\cJ_{1}=\sum_{i,j:\,\hat{\kappa}_{ij}=0} \la(1-s_{r(ij)}^{2})^{\delta_{r(ij)}}\prod_{e\in E(ij)\cap \widehat{E}} \eta_{e} \ra$. However, if we define
\begin{equation}\label{eq:monmom2}
\cJ_{2}=\sum_{i=1}^{k}\sum_{w,w'\in L_{i}} \la (1-s_{r(ww')}^{2})^{\delta_{r(ww')}}\!\!\!\prod_{e\in E(ww')} \!\!\!\eta_{e}\ra
\end{equation}
then it can be checked that $\cJ_{1}=\cJ_{2}$. To show this fix ${j}=1,\ldots, k$ and $w,w'\in L_{j}$. Note that by construction each of $w,w'$ either has degree two in $\widehat{T}$ or is a leaf of $T$. Hence by the definition of $\widehat{E}$ there exist $i,j\in [n]$ such that $E(ij)\cap\widehat{E}=E(ww')$. It follows that each generator in (\ref{eq:monmom2}) is also in the set of generators of $\cJ_{1}$ and hence $\cJ_{2}\subseteq \cJ_{1}$. To show the opposite inclusion note that if $E(ij)$ intersects with more than one component $T_{1},\ldots, T_{k}$ then the corresponding generator in (\ref{eq:monmom1}) is a product of some generators in (\ref{eq:monmom2}) and hence it lies in $\cJ_{2}$. 

Since the generators of every $\cJ(S_{i})$ for $i=1,\ldots,m$ and every {$$\sum_{w,w'\in L_{j}} \la (1-s_{r(ww')}^{2})^{\delta_{r(ww')}}\prod_{e\in E(ww')} \eta_{e}\ra$$} for $j=1,\ldots,k$ involve disjoint sets of variables then by Proposition \ref{lem:product-rlct} the term in  (\ref{eq:rlctredstep}) is equal to 
$$
\sum_{i=1}^{m}{\rm RLCT}_{\omega_{0}}(\cJ(S_{i}))+ \sum_{i=1}^{k}{\rm RLCT}_{\omega_{0}}(\!\!\!\sum_{w,w'\in L_{j}} \la (1-s_{r(ww')}^{2})^{\delta_{r(ww')}}\!\!\!\prod_{e\in E(ww')} \!\!\!\eta_{e}\ra)+(0,1-m-k). 
$$
It can be easily checked that by Proposition \ref{lem:szacowanie} (iv) for each $i=1,\ldots,k$
$$
{\rm RLCT}_{\omega_{0}}(\!\!\!\sum_{w,w'\in L_{i}} \la (1-s_{r(ww')}^{2})^{\delta_{r(ww')}}\!\!\!\prod_{e\in E(ww')} \!\!\!\eta_{e}\ra)={\rm RLCT}_{\omega_{0}}(\cJ(T_{i}))
$$
which finishes the proof.
\end{proof}

\section{The case of zero covariances}\label{sec:zerocase}

In this subsection we assume that $\hat{\kappa}_{ij}=0$ for all $i,j\in [n]$. The aim  is to prove the following proposition.
\begin{prop}\label{prop:zerocase}
Let $T$ be a trivalent tree with $n\neq 3$ leaves rooted in $r\in V$. Let $\hat{p}\in \cM_{T}$ be such that $\hat{\kappa}_{ij}=0$ for all $i,j\in [n]$. Let $\cJ=\sum_{i,j\in [n]}\la\kappa_{ij}(\omega)\ra$. Then  
$$\min_{\omega_{0}\in \widehat{\Omega}_{T}}{\rm RLCT}_{\omega_{0}}(\cJ)=(\frac{n}{4},m),$$ 
where $m=1$ if either $r$ is a leaf of $T$ or $r$ together with all its neighbors are all inner nodes of $T$. In all other cases we cannot obtain an explicit bound for $m$ and hence  $m\geq 1$.
\end{prop}

There is no coincidence in the fact that $\cJ$ here denotes $\sum_{i,j\in [n]}\la\kappa_{ij}(\omega)\ra$ and in Section \ref{sec:splittrees} it denotes $\sum_{I\in[n]_{\geq 2}}\la\kappa_{I}(\omega)-\hat{\kappa}_{I}\ra$. In fact if $\hat{\kappa}_{ij}=0$ for all $i,j\in[n]$ then these two ideals are equal (see the beginning of the proof of Lemma \ref{lem:redzQ}).

The strategy of the proof of Proposition \ref{prop:zerocase} is as follows. First in Section \ref{sec:deepest} we show that the local computations can be restricted to a special subset of $\widehat{\Omega}_{T}$ over which  $\cJ$ can be replaced by a monomial ideal. Then in Section \ref{sec:n-diagram} we present the method to compute real log-canonical threshold of a monomial ideal. We use this method in Section \ref{sec:zeroproof}.

\subsection{The deepest singularity}\label{sec:deepest}

First we note that the ideal $\cJ$ in Proposition \ref{prop:zerocase} depends on $s_{v}$ for $v\in V$ only through the value of $s_{v}^{2}$. It follows that the computations can be reduced only to points satisfying $s_{v}\geq 0$ for all inner nodes $v$ of $T$. Henceforth in this section we always assume this is the case. We define the \textit{deepest singularity} of $\widehat{\Omega}_{T}$ as
\begin{equation}\label{eq:deepest}
\widehat{\Omega}_{{\rm deep}}:=\{\omega\in\widehat{\Omega}_T: \eta_{e}=0 \,\,\mbox{for all } e\in \widehat{E}, \,\,\,s_{v}=1 \,\,\mbox{for all } v\in \widehat{V}\}.
\end{equation} 
We note that since $\hat{\kappa}_{ij}=0$ for all $i,j\in [n]$ then $\widehat{E}=E$ and $\widehat{V}$ is equal to the set of all inner nodes of $T$ and $\widehat{\Omega}_{{\rm deep}}$ is an affine subspace constrained to $\Omega_{T}$. 
\begin{prop}\label{lem:red-2-deepest}
Let $T$ be a tree with $n$ leaves. Let $\hat{p}\in\cM_{T}$ such that $\hat{\kappa}_{ij}=0$ for all $i,j\in [n]$. Then 
\begin{equation}\label{eq:mindeep}
\min_{\omega_{0}\in\widehat{\Omega}_{T}} {\rm RLCT}_{\omega_{0}}( \cJ)=\min_{\omega_{0}\in\widehat{\Omega}_{{\rm deep}}} {\rm RLCT}_{\omega_{0}}( \cJ).
\end{equation}
\end{prop}
\begin{proof}
We build on the proof of Theorem 5.8 in \cite{pwz2010-identifiability}. We first show that $\widehat{\Omega}_{T}$ is a union of affine subspaces constrained to $\Omega_{T}$ with a common intersection given by $\widehat{\Omega}_{{\rm deep}}$. Let $V_{0}\subseteq \widehat{V}$ and $E_{0}\subseteq \widehat{E}$ and
\begin{equation}\label{eq:omega-AB}
\Omega_{(V_0,E_0)} =\{\omega\in \widehat{\Omega}_{T}: \,\,s_{v}=1 \mbox{ for all } v\in V_{0}, \,\,\eta_{u,v}=0\mbox{ for all } (u,v)\in E_{0} \}.
\end{equation}
We say that $(V_0,E_0)$ is \textit{minimal for} $\widehat{\Sigma}$ if for  every point $\omega$ in $\Omega_{(V_0,E_0)}$ and for every $i,j\in [n]$ $\kappa_{ij}(\omega)=0$  and furthermore that $(V_0,E_0)$ is minimal with such a property (with respect to inclusion on both coordinates). We now show that the $\hat{p}$-fiber satisfies 
\begin{equation}\label{eq:p-fiber-decomp}
\widehat{\Omega}_{T}=\bigcup_{(V_0,E_0)\mbox{ min.}} \Omega_{(V_0,E_0)}.
\end{equation}
The first inclusion ``$\subseteq$'' follows from the fact that if $\omega \in \widehat{\Omega}_{T}$ then $\kappa_{ij}(\omega)=\hat{\kappa}_{ij}=0$ for all $i,j\in [n]$. Therefore $\omega\in \Omega_{(V_0,E_0)}$ for some minimal $(V_0,E_0)$. The second inclusion is obvious. 

 Each $\Omega_{(V_0,E_0)}$ is a union of $2^{|V_{0}|}$ an affine subspace in $\R^{|V|+|E|}$, denoted by $M_{(V_0,E_0)}$,  constrained to $\Omega_{T}$. Let $\cS$ denote the intersection lattice of all $M_{(V_{0},E_{0})}$ for  $(V_{0},E_{0})$ minimal (c.f. Section 3.1 in \cite{macpherson}) with ordering denoted by $\leq$. For each $i\in \cS$ let $M^{(i)}$ denote the corresponding intersection and define
\begin{equation}\label{eq:stratification}
S_{i}=M^{(i)}\setminus\bigcup_{j<i} M^{(j)}.
\end{equation}
In this way we obtain an \textit{$\cS$-induced  decomposition} of $\R^{|V|+|E|}$. 
 
By \cite[Example 9.3.17]{lazarsfeld2004positivity} the function $\omega\mapsto {\rm rlct}_{\omega}(f)$ is lower semicontinuous (the argument used there works over the real numbers). This means that for every $\omega_{0}\in \Omega_{T}$ and $\e>0$ there exists a neighborhood $U$ of $\omega_{0}$ such that ${\rm rlct}_{\omega_{0}}(f)\leq {\rm rlct}_{\omega}(f)+\e$ for all $\omega\in U$. Since the set of values of the real log-canonical threshold is discrete this means that for every $\omega_{0}\in\widehat{\Omega}_{T}$ and any sufficiently small neighborhood $W_{0}$ of $\omega_{0}$ one has ${\rm rlct}_{\omega_{0}}(f)\leq {\rm rlct}_{\omega}(f)$  for all $\omega\in W_{0}$. Since for any neighborhood $W_{0}$ of $\omega_{0}\in\widehat{\Omega}_{{\rm deep}}$ we have $W_{0}\cap S_{i}\neq \emptyset$ for all $i\in\cS$ then necessarily the minimum of the real log-canonical threshold is attained for a point from the deepest singularity. 
\end{proof}

Proposition \ref{lem:red-2-deepest} shows that in the singular case we can restrict our analysis to the neighborhood of $\widehat{\Omega}_{{\rm deep}}$. Often however we also consider points in a bigger set
$$\widehat{\Omega}_{0}=\{\omega\in\widehat{\Omega}_{T}:\eta_{u,v}=0\,\,\mbox{for all } (u,v)\in \widehat{E}\}.$$
Note that $\widehat{\Omega}_{{\rm deep}}$ lies on the boundary of $\Omega_{T}$ but $\widehat{\Omega}_{0}$ also contains internal points of $\Omega_{T}$ which will be crucial for some of the arguments later.

\begin{lem}\label{lem:jotplusrest}\label{lem:Jjaksiepatrzy}
Assume that $\hat{p}\in\cM_{T}$ is such that $\hat{\kappa}_{ij}=0$ for all $i,j\in [n]$. Let \mbox{$\cJ=\sum_{i,j\in[n]}\la\kappa_{ij}(\omega)\ra$} and $\cJ(\omega_{0})$ be the ideal $\cJ$ translated to the origin. Then for every $\omega_{0}\in \widehat{\Omega}_{0}$
\begin{equation}\label{eq:rlct-jj}
{\rm RLCT}_{0}(\cJ({\omega_{0}}))={\rm RLCT}_{0}(\cJ'),
\end{equation}
 where  $\cJ'$ is a monomial ideal such that each $\kappa_{ij}(\omega+\omega_{0})$ in the set of generators of $\cJ({\omega_{0}})$ is replaced either by
 $$
 \begin{array}{ll}
 s_{r(ij)}\prod_{(u,v)\in E(ij)}\eta_{u,v} & \mbox{if } s_{r(ij)}^{0}= 1, \mbox{ or by}\\
\prod_{(u,v)\in E(ij)}\eta_{u,v}&\mbox{if } s_{r(ij)}^{0}\neq 1.
\end{array}
 $$
\end{lem}
\begin{proof}
Let $i,j\in [n]$ and assume that $\omega_{0}\in \widehat{\Omega}_{0}$ then by Proposition \ref{lem:param}
$$
\kappa_{ij}(\omega+\omega_{0})=\frac{1}{4}(1-(s_{r(ij)}+s^{0}_{r(ij)})^{2})\prod_{e\in E(ij)}\eta_{e}.
$$
If $s_{r(ij)}^{0}\neq 1$ for a sufficiently small $\e>0$ there exist positive constants $c(\e)$, $C(\e)$ such that  $c(\e)<1-(s_{r(ij)}+s_{r(ij)}^{0})^{2}<C(\e)$ for $s_{r(ij)}\in (-\e,\e)$. Therefore by Proposition \ref{lem:szacowanie} (iv) we can replace this term in (\ref{eq:rlct-jj}) with $1$. If $s_{r(ij)}^{0}=1$ rewrite $1-(1+s_{r(ij)}^{2})^{2}$ as $-s_{r(ij)}(2+s_{r(ij)})$. For a sufficiently small $\e$ we can find two positive constants $c(\e),C(\e)$ such that $c<2+s_{r(ij)}<C$ whenever $s_{r(ij)}\in (-\e,\e)$. Again by Proposition \ref{lem:szacowanie} (iv) we can replace this term with $1$. This proves equation (\ref{eq:rlct-jj}).
\end{proof}
Since $\cJ'$ is a monomial ideal then by Corollary 5.3 in \cite{shaowei_rlct} we can compute ${\rm RLCT}_{{0}}(\cJ' )$ using the method of Newton diagrams. We present this method in the following subsection.

\subsection{Newton diagram method}\label{sec:n-diagram} Given an analytic function $f\in\cA_{0}(\R^{d})$ we pick local coordinates $x=(x_{1},\ldots, x_{d})$ in a neighborhood of the origin. This allows us to represent $f$ as a power series in $x_{1},\ldots, x_{d}$ such that $f(x)=\sum_{\a} c_{\a} x^{\a}$. The exponents of terms of the polynomial $f$ are vectors in $\N^{d}$. The \textit{Newton polyhedron} of $f$ denoted by $\Gamma_+(f)$ is the convex hull of the subset
 $$
 \{\a+\a':\, c_{\a}\neq 0, \a'\in\R_{\geq 0}^{d}\}.
 $$
A subset $\gamma\subset \Gamma_{+}(f)$ is a \textit{face} of $\Gamma_{+}(f)$ if there exists $\beta\in\R^{d}$ such that
$$
\gamma=\{\a\in\Gamma_{+}(f):\, \la\a,\beta\ra\leq \la\a',\beta\ra\,\mbox{ for all } \a'\in\Gamma_{+}(f)\}.
$$
If $\gamma$ is a subset of $\Gamma_{+}(f)$ then we define $f_{\gamma}(x)=\sum_{\a\in \gamma\cap\N^{d}} c_{\a} x^{\a}$. The \textit{principal part} of $f$ is by definition the sum of all terms of $f$ supported on all compact faces of $\Gamma_{+}(f)$.
\begin{exmp}
Let $f(x,y)=x^{3}+2xy+6x^{2}y+3x^{4}y+y^{2}$. Then the Newton diagram looks as follows
\begin{center}
\includegraphics{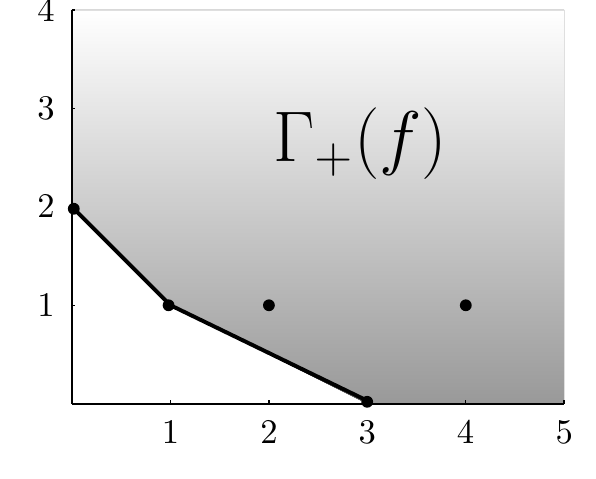}
\end{center}
where the dots correspond to the terms of $f$. There are only two bounded facets of $\Gamma_{+}(f)$ and the principal part of $f$ is equal to $f(x,y)=x^{3}+xy+y^{2}$.
\end{exmp}
\begin{defn}\label{def:nondegenerate}
The principal part of the power series $f$ with real coefficients  is $\R$-\textit{nondegenerate} if for all compact faces $\gamma$ of $\Gamma_+(f)$
\begin{equation}\label{eq:non-deg}
\left\{x\in\R^{n}:\,\frac{\partial f_\gamma}{\partial x_{1}}(x)=\cdots=\frac{\partial f_\gamma}{\partial x_{n}}(x)=0\right\}\subseteq \left\{\omega\in\R^{n}:x_{1}\cdots x_{n}=0\right\}.
\end{equation}
From the geometric point of view this condition means that the singular locus of the hypersurface defined by $f_{\gamma}(x)=0$  lies outside of $(\R^{*})^{n}$ for all compact faces $\gamma$ of $\Gamma_{+}(f)$.
\end{defn}

The following theorem shows that if the principal part of $f$ is $\R$-nondegenerate and $f\in\cA^{\geq}_{\Theta}$ it greatly facilitates the computations in Theorem \ref{th:watanabe}. An example of an application of these methods in statistical analysis can be found in \cite{watanabe_newton}. 
\begin{thm}[Theorem 5.6, \cite{shaowei_rlct}]\label{thm:arnold}
Let  $f\in\cA^{\geq}_{{0}}(\R^{d})$ and $f({0})=0$. If the principal part of $f$ is $\R$-nondegenerate then ${\rm RLCT}_{{0}}(f)=(\frac{1}{t},c)$  where $t$ is the smallest number such that the vector $(t,\ldots,t)$ hits the polyhedron $\Gamma_+({f})$ and $c$ is the codimension of the face it hits.
\end{thm}

Let now $f\in \cA^{\geq}_{\theta_{0}}$ such that $f(\theta_{0})=0$. We can then center $f$ at $\theta_{0}$ obtaining a function in $\cA_{0}^{\geq}$. If $f$ is nonnegative then we can use Theorem \ref{thm:arnold} to compute ${\rm RLCT}_{\theta_{0}}(f)$. Note that this theorem in general will not give us ${\rm RLCT}_{\Theta_{0}}(f)$ if ${0}$ is a boundary point of $\Theta$. For a discussion see \cite[Section 8.3.4]{arnold1985singularities} and Example 2.7 in \cite{shaowei_rlct}.

\subsection{Proof of Proposition \ref{prop:zerocase}}\label{sec:zeroproof}
Let $n\geq 4$.  For each $\omega_{0}\in \widehat{\Omega}_{0}$ let $\delta=\delta(\omega_{0})\in\{0,1\}^{V}$ denote the indicator vector satisfying   $\delta_{v}=1$ if $s_{v}^{0}=1$ and $\delta_{v}=0$ otherwise. In particular $\delta_{i}=0$ for all $i=1,\ldots, n$ because the leaves are assumed to be non-degenerate. Let $\cV_{\delta}=\R^{n_{e}+|\delta|}=\R^{|\delta|}\times \R^{n_{e}}$ be a real space with variables representing the edges $(x_{e})_{e\in E}$ and nodes $(y_{v})$ for all $v$ such that $\delta_{v}=1$. With some arbitrary numbering of the nodes and edges we order the variables as follows: $y_1\prec\cdots\prec y_{|\delta|}\prec x_1\prec\cdots \prec x_{n_{e}}$. In Lemma \ref{lem:jotplusrest} for each $\omega_{0}\in \widehat{\Omega}_{0}$ we reduced our computations to the analysis of ${\rm RLCT}_{0}(\cJ')$ where $\cJ'$ has a simple monomial form. Let $Q_{\delta}$ be a polynomial on $\Omega_{T}$ defined as a sum of squares of generators of $\cJ'$. In particular ${\rm RLCT}_{0}(\cJ')={\rm RLCT}_{0}(Q_{\delta})$. The exponents of terms of the polynomial $Q_{\delta} (\omega)$ are vectors in $\{0,2\}^{n_{e}+|\delta|}$. We have that
\begin{equation}\label{eq:Q1}
Q_{\delta} (\omega)=\sum_{i,j\in[n]}s_{r(ij)}^{2\delta_{r(ij)}}\prod_{(u,v)\in E(ij)} \eta_{u,v}^{2}.
\end{equation}

 If $f$ is a polynomial then the convex hull of the exponents of the terms in the sum is called the \textit{Newton polytope} and denoted $\Gamma(f)$. Since each term of $Q_{\delta}$ corresponds to a path between two leaves then the construction of the Newton polytope $\Gamma(Q_{\delta} )\subset\cV_{\delta}$ gives a direct relationship between paths in $T$ and points generating the polytope. Convex combinations of points corresponding to paths give rise to points in the polytope. Let $E_{0}\subseteq E$ be the subset of edges of $T$ such that one of the ends is in the set of leaves of $T$. We call these edges \textit{terminal}. Note that each point generating $\Gamma(Q_{\delta} )$ satisfies $\sum_{e\in E_{0}} x_{e}=4$. This follows from the fact that each of these points corresponds to a path between two leaves in $T$ and every such a path need to cross exactly two terminal edges. Consequently each point of $\Gamma(Q_{\delta} )$ needs to satisfy this equation as well. The induced facet of the Newton polyhedron $\Gamma_{+}(Q_{\delta} )$ is given as   
\begin{equation}\label{eq:F0}
F_{0}=\{(\mathbf{y},\mathbf{x})\in \Gamma_{+}(Q_{\delta} ): \sum_{e\in E_{0}} x_{e}= 4\}
\end{equation}
and each point of $\Gamma_{+}(Q_{\delta} )$ satisfies $\sum_{e\in E_{0}} x_{e}\geq 4$.

The following lemma proves a part of Proposition \ref{prop:zerocase}.
\begin{lem}[The real log-canonical threshold of $\cI$]\label{lem:lambda}
Let $T$ be a trivalent tree with $n$ leaves where $n\neq 3$. If  $\omega_{0}\in \widehat{\Omega}_{0}$ then
${\rm rclt}_{{0}}(\cJ')=\frac{n}{4}$. \end{lem}
\begin{proof}
If $n=2$ then since $s_{1}^{0},s_{2}^{0}\neq 1$ by Lemma \ref{lem:Jjaksiepatrzy} we have that
$${\rm RLCT}_{\omega_{0}}(\cJ)={\rm RLCT}_{\omega_{0}}(\la \kappa_{12}(\omega)\ra)={\rm RLCT}_{0}(\eta_{12}^{2})=(\frac{1}{2},1).$$
Therefore Proposition \ref{prop:zerocase} holds in this case. Now assume that $n\geq 4$. By Theorem \ref{thm:arnold} we have to show that $t=4/n$ is the smallest $t$ such that the vector $(t,\ldots, t)$ hits $\Gamma_{+}(Q_{\delta} )$. To show that $4/n\mathbf{1}\in \Gamma_{+}(Q_{\delta} )$ we construct a point $q\in \Gamma(Q_{\delta} )$ such that $q\leq 4/n\mathbf{1}$ coordinatewise. The point is constructed as follows.  
\begin{const}\label{def:constr_net} Let $T=(V,E)$ be a trivalent rooted tree with $n\geq 4$ leaves. We present two constructions of networks of paths between the leaves of $T$. 

The first construction is for the case when $\delta_{r}=1$. In this case in particular $T$ is rooted in an inner node. If $n=4$ then the network consists of the two  paths within cherries counted with multiplicity two. 
\begin{center}
\includegraphics[scale=0.7]{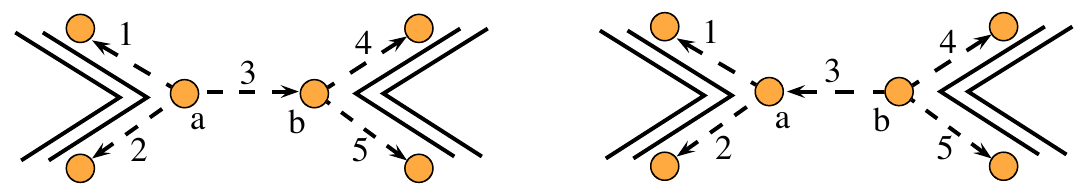}
\end{center}
Each of the paths corresponds to a point in $\Gamma(Q_{\delta} )$. We order the coordinates of $\cV_{\delta}=\R^{5+|\delta|}$ by $y_{a}\prec y_{b}\prec x_{1}\prec \cdots\prec x_{5}$ where $y_{a},y_{b}$ are included only if $\delta_{a},\delta_{b}=1$. For example the point corresponding to the path involving edges $e_{1}$ and $e_{2}$ is $(2, 0; 2, 2, 0 , 0 , 0)$. The barycenter of the points corresponding to all the four paths in the network is $(1,1;1,1,0,1,1)$ both if $T$ is rooted in $a$ or $b$. 

If $n> 4$ then we build the network recursively. Assume that $T$ is rooted in an inner node $a$ and pick an inner edge $(a,b)$. Label the edges incident with $a$ and $b$ as for the quartet tree above and consider the subtree given by the quartet tree. Draw four paths as on the picture above. Let $v$ be any leaf of the quartet subtree which is not a leaf of $T$ and label the two additional edges incident with $v$ by $e_{6}$ and $e_{7}$. Now we extend the network by adding $e_{6}$ to one of the paths terminating in $v$ and $e_{7}$ to the other. Next we add an additional path involving only $e_{6}$ and $e_{7}$ like on the picture below. By construction $v$ is the root of the additional path. We extend the network cherry by cherry until it covers all terminal edges. 
\begin{center}
\includegraphics[scale=0.7]{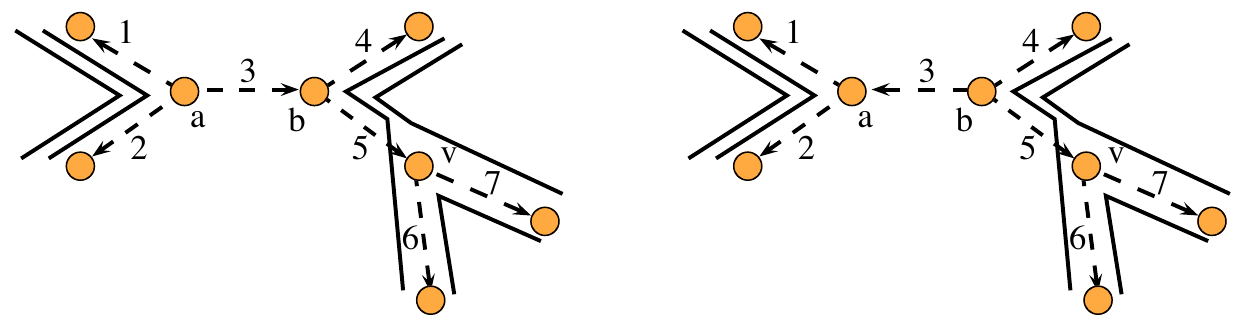}
\end{center}

Note that we have made some choices building up the network and hence the construction is not unique. However, each of the inner nodes  is always a root of at least one and at most two paths. Moreover, each edge is covered at most twice and each terminating edge is covered exactly two times. We have $n$ paths in the network, all representing points of $\Gamma({Q}_{\delta})$ denoted by $q_{1},\ldots, q_{n}$. Let $q=\frac{1}{n}\sum_{i=1}^{n} q_{i}$ then $q\in\Gamma(Q_{\delta} )$ is given by $x_{ab}=0$, $x_{e}=4/n$ for all $e\in E\setminus (a,b)$. The other coordinates by construction satisfy $y_{a}=4/n$, $y_{b}=4/n$ if $\delta_{b}=1$, and $y_{v}= 2/n$ for all $v\in V\setminus\{a,b\}$ such that $\delta_{v}=1$. 

If $\delta_{r}=0$ then we proceed as follows. For $n=4$ consider a network of all the possible paths all counted with multiplicity one apart from the cherry paths (paths of length two) counted with multiplicity two. This makes eight paths and each edge is covered exactly four times. With the order of the coordinates as above the coordinates of the point representing the barycenter of all paths in the network satisfy $x_{e}=1$ for all $e\in E$ and $y_{v}=1/2$ for all $v$ such that $\delta_{v}=1$. This construction generalizes recursively in a similar way as the one for $T$ rooted in an inner node. We always have $2n$ paths and each edge is covered exactly four times. The network induces a point ${q}\in\Gamma(Q_{\delta} )$ with coordinates given by $y_{v}=2/n$ for all $v\in V$ such that $\delta_{v}=1$ and $x_{e}=4/n$ for $e\in E$. (This finishes the construction.)
\end{const} 
The point  $4/n \,\mathbf{1}$ lies in $\Gamma_{+}(Q_{\delta} )$. This follows from Construction \ref{def:constr_net} and the fact that the constructed point $q\in \Gamma(Q_{\delta} )$ satisfies $q\leq \frac{4}{n}\,\mathbf{1}$. Moreover, for any $s<4/n$ the point $s (1,\ldots,1)$ does not satisfy $\sum_{e\in E_{0}} x_{e}\geq 4$ and hence it cannot be in $\Gamma_{+}(Q_{\delta} )$. It follows that $4/n$ is the smallest $t$ such that $t\mathbf{1}\in \Gamma_{+}(Q_{\delta})$ and therefore ${\rm rlct}_{0}(\cJ')=n/4$. We note that the result does not depend on $\delta$. 

\end{proof}


To compute the multiplicity  of the real log-canonical threshold of $Q_{\delta}$ we have to get a better understanding of the polyhedron $\Gamma_{+}(Q_{\delta} )$. According to Theorem \ref{thm:arnold} we need to find the codimension of the face of $\Gamma_{+}(Q_{\delta} )$ hit by $4/n\mathbf{1}$. First we find the hyperplane representation of the Newton polytope $\Gamma(Q_{\delta} )$ reducing the problem to a simpler but equivalent one. 

\begin{defn}[A pair-edge incidence polytope]
Let $T=(V,E)$ be a trivalent tree with $n\geq 4$ leaves. We define a polytope $P_{n}\subset\R^{n_{e}}$, where ${n_{e}}=2n-3$, as a convex combination of points $(q_{ij})_{i,j\in [n]}$ where $k$-th coordinate of $q_{ij}$ is one if the $k$-th edge is in the path between $i$ and $j$ and there is zero otherwise. We call $P_{n}$ a \textit{pair-edge incidence polytope} by analogy to the pair-edge incidence matrix defined by Mihaescu and Pachter \cite[Definition 1]{mihaescu2008combinatorics}. 
\end{defn}

The reason to study the pair-edge incidence polytope is that its structure can be handled easily and this can be shown to be affinely equivalent to $\Gamma(Q_{\delta} )$. This is immediate if $\delta=(0,\ldots,0)$ since $Q_{\textbf{0}}=2P_{n}$. For an arbitrary $\delta$ fix a rooting $r$ of $T$ and define a linear map $f_{r}:\R^{n_{e}}\rightarrow \R^{|\delta|}$ as follows. For each $v\in V\setminus r$ such that $\delta_{v}=1$ set $$y_{v}=1/2(x_{v\ch_{1}(v)}+x_{v\ch_{2}(v)}-x_{\pa(v)v}),$$ where  $\ch_{1}(v)$, $\ch_{2}(v)$ denotes the two children of $v$. If $\delta_{r}=1$ then set 
$$y_{r}=1/2(x_{r\ch_{1}(r)}+x_{r\ch_{2}(r)}+x_{r\ch_{3}(r)}).$$ It can be easily checked that for a map $(\id\times f_{r}):\R^{n_{e}}\rightarrow\R^{n_{e}}\times \R^{|\delta|}$ one has  $(\id\times f_{r})(2P_{n})=\Gamma(Q_{\delta} )$. This follows from the fact that for each point $y_{r}=2$ if and only if the path crosses $r$ and for any other node $y_{v}=2$ if and only if the path crosses $v$ and $v$ is the root of the path, i.e. if the path crosses both children of $v$.

\begin{lem}\label{lem:structureP}
Let $P_{n}\subset \R^{n_{e}}$ be a pair-edge incidence polytope for a trivalent tree with $n$ leaves where $n\geq 4$. Then $\dim(P_{n})=n_{e}-1=2n-4$. Hence the codimension of $P_{n}$ is one. The unique equation defining $P_{n}$ is $\sum_{e\in E_{0}} x_{e}=2$. For each inner node $v\in V$ let $e_{1}(v)$, $e_{2}(v)$, $e_{3}(v)$ denote the three adjacent edges. Then exactly $3(n-2)$ facets define $P_{n}$ and they are given by
\begin{eqnarray}\label{eq:facets}
x_{e_{1}(v)}+x_{e_{2}(v)}-x_{e_{3}(v)}\geq 0,\quad x_{e_{2}(v)}+x_{e_{3}(v)}-x_{e_{1}(v)}\geq 0,\\ 
\nonumber \mbox{and }x_{e_{3}(v)}+x_{e_{1}(v)}-x_{e_{2}(v)}\geq 0\quad\mbox{for all $v\in V$.}
\end{eqnarray}
\end{lem}
\begin{proof}
Let $M_{n}$ be the pair-edge incidence matrix, i.e. a ${n \choose 2} \times n_{e} $ matrix with rows corresponding to the points defining $P_{n}$. By Lemma 1 in \cite{mihaescu2008combinatorics} the matrix has full rank and hence $P_{n}$ has codimension one in $\R^{n_{e}}$. Moreover since each path necessarily crosses two terminal edges then each point generating $P_{n}$ satisfies the equation $\sum_{e\in E_{0}} x_{e}=2$ and hence this is the equation defining the affine subspace containing $P_{n}$. 

Now we show that the inequalities give a valid facet description for $P_{n}$. This can be checked directly for $n=4$ (e.g. using Polymake \cite{math.CO/0507273}). Assume this is true for all $k<n$. By $Q_{n}$ we will define the polytope defined by the equation $\sum_{e\in E_{0}} x_{e}=2$ and $3(n-2)$ inequalities given by (\ref{eq:facets}). It is obvious that $P_{n}\subseteq Q_{n}$ since all points generating $P_{n}$ satisfy the equation and the inequalities. We show that the opposite inclusion also holds. 

Consider any cherry $\{e_{1},e_{2}\}\subset E$ in the tree given by two leaves denoted by $1$, $2$ and the separating inner node $a$. Note that the inequalities in (\ref{eq:facets}) imply in particular that $x_{e}\geq 0$ for all $e\in E$. Define a projection $\pi: \R^{n_{e}}\rightarrow \R^{n_{e}-2}$ on the coordinates related to all the edges apart from the two in the cherry. We have $\pi(Q_{n})= \widehat{Q}_{n-1}$, where $\widehat{P}={\rm conv}\{0,P\}$ is a cone with the base given by $P$.  The projection  $\pi(Q_{n})$ is described by all the triples of inequalities for all the inner nodes apart from the one incident with the cherry and the defining equation becomes an inequality
$$
\sum_{e\in E_{0}\setminus \{e_{1},e_{2}\}}x_{e}\leq 2.
$$
Denote the edge incident with $e_{1},e_{2}$ by $e_{3}$ and the related coordinates of $x$ by $x_{1},x_{2},x_{3}$. The three inequalities involving $x_{1}$ and $x_{2}$ do not affect the projection since they imply that 
$$
\max\{x_{1}-x_{2},x_{2}-x_{1}\}\leq x_{3}\leq x_{1}+x_{2}
$$
and hence in particular if $x_{1}=x_{2}$ the constraint becomes $[0,2x_{1}]$. Consequently the set given by $x_{1}+x_{2}-x_{3}\geq 0$, $x_{1}+x_{3}-x_{2}\geq 0$, $x_{2}+x_{3}-x_{1}\geq 0$ projects down to $\R_{\geq 0}$. However since $\widehat{Q}_{n-1}$ is contained in the nonnegative orthant there are no additional constraints on $x_{3}$. Inequalities in Equation (\ref{eq:facets}) define a polyhedral cone and the equation $\sum_{e\in E_{0}\setminus \{e_{1},e_{2}\}}x_{e}=t$ for $t\geq 0$ cuts out a bounded slice of the cone which is equal to $t\cdot P_{n-1}$. The sum of all these for $t\in[0,2]$ is exactly $\widehat{Q}_{n-1}$. Since $\widehat{Q}_{n-1}=\widehat{P}_{n-1}$ by induction then each $\pi(x)$ is a convex combination of the points generating $P_{n-1}$ and zero, i.e. $\pi(x)=\sum c_{ij} p_{ij}$ where the sum is over all $i\neq j\in \{a,3,\ldots,n\}$ and $c_{ij}\geq 0$, $\sum c_{ij}\leq 1$. 

Let $x\in Q_{n}$. Since $\pi(x)\in \widehat{P}_{n-1}$ we can write it as the linear combination above. Next we lift this combination back to $Q_{n}$ and show that any such a lift has to lie in $P_{n}$. This would imply that in particular $x\in P_{n}$. Let $y$ denote a lift of $\pi(x)$ to $Q_{n}$. We have $$y=\sum c_{ij} r_{ij}+\left(1-\sum c_{ij}\right) \,r_{0},$$
where $r_{ij}$ is a lift of $\pi(p_{ij})$ and $r_{0}$ is a lift of the origin. It suffices to show that each $r_{ij}$ and $r_{0}$ necessarily lie in $P_{n}$.

 Consider the following three cases. First if $p_{ij}\in P_{n-1}$ is such that $x_{3}=0$ then sum of all the other coordinates related to the terminal edges is two since $P_{n-1}=Q_{n-1}$ and $Q_{n-1}$ satisfy the equation $\sum_{e\in E_{0}\setminus \{e_{1},e_{2}\}}x_{e}=2$. Hence if we lift $\pi(p_{ij})$ to $Q_{n}$ then $x_{3}=0$ and $$x_{1}+x_{2}\geq 0 ,\quad  x_{1}-x_{2}\geq 0,\quad x_{2}-x_{1}\geq 0$$
by plugging $x_{3}=0$ into the three inequalities for the node $a$. But since $r_{ij}$ must also satisfy the equation $\sum_{e\in E_{0}} x_{e}=2$ and since we already have 
$$\sum_{e\in E_{0}\setminus \{e_{1},e_{2}\}}x_{e}=2$$ then $x_{1}+x_{2}=0$ and hence  $x_{1}=x_{2}=0$. Consequently, $r_{ij}$ is a vertex of $P_{n}$. Second if $p_{ij}$ is a vertex of $P_{n-1}$ such that $x_{3}=1$ then the sum of all the other coordinates of $p_{ij}$ related to the terminal edges is one and hence since the lift is in $Q_{n}$ we have $x_{1}+x_{2}=1$. The additional inequalities give that $x_{1},x_{2}\geq 0$. Hence in this case $r_{ij}$ is a convex combination of two points in $P_{n}$ - one corresponding to a path finishing in one of the edges and the other in the other. Finally, we can easily check that zero lifts uniquely to a point in $P_{n}$ corresponding to the path $E(12)$. Indeed, from the equation defining $Q_{n}$ we have $x_{1}+x_{2}=2$ and from the inequalities since $x_{3}=0$ we have $x_{1}=x_{2}=1$. Therefore every lift $y$ of $\pi(x)$ to $Q_{n}$  can be written as a convex combination of points generating $P_{n}$ and hence $y\in P_{n}$.  Consequently $x\in P_{n}$ and hence $Q_{n}\subseteq P_{n}$.
\end{proof}

Lemma \ref{lem:structureP} shows that $P_{n}$ has an extremely simple structure. The inequalities give a polyhedral cone and the equation cuts out the polytope $P_{n}$ as a slice of this cone. The result gives us also the representation of $\Gamma(Q_{\delta} )$ in terms of the defining equations and inequalities. 

\begin{prop}[Structure of $\Gamma(Q_{\delta} )$]\label{lem:structureG}
Polytope $\Gamma(Q_{\delta} )\subset \cV_{\delta}$ is given as an intersection of the sets defined by the inequalities in (\ref{eq:facets}) together with $|\delta|+1$ equations given by  
\begin{equation}\label{eq:eqGQ1}
\begin{array}{l}
2y_{v}=x_{v\ch_{1}(v)}+x_{v\ch_{2}(v)}-x_{\pa(v)v} \quad\mbox{ for all } v\neq r \mbox{ such that } \delta_{v}=1,\\
2y_{r}=x_{r\ch_{1}(r)}+x_{r\ch_{2}(r)}+x_{r\ch_{3}(r)}\quad\mbox{if }\delta_{r}=1,\mbox{ and}\\
\sum_{e\in E_{0}} x_{e}=4.
\end{array}
\end{equation}
\end{prop}

From this we can partially understand the structure of $\Gamma_{+}(Q_{\delta} )$. First note that $\Gamma_{+}(f)=\Gamma(f)+\R^{d}_{\geq 0}$, where the plus denotes the Minkowski sum. The \textit{ Minkowski sum} of two polyhedra is by definition
$$
\Gamma_{1}+\Gamma_{2}=\{x+y\in\R^{d}: \,x\in\Gamma_{1}, y\in \Gamma_{2}\}.
$$
\begin{lem}\label{lem:positivecoef}
Let $\Gamma\subset \R^{n}_{\geq 0}$ be a polytope and let $\Gamma_{+}$ be the Minkowski sum of $\Gamma$ and the standard cone $\R_{\geq 0}^{n}$. Then all the facets of $\Gamma_{+}$ are of the form $\sum_{i} a_{i} x_{i}\geq c$ where $a_{i}\geq 0$ and $c\geq 0$.
\end{lem}

Now we are ready to compute multiplicities of the real log-canonical threshold ${\rm RLCT}_{{0}}(Q_{\delta} )$. This completes the proof of Proposition \ref{prop:zerocase}.
\begin{lem}[Computing multiplicities]\label{lem:multiplicities}
Let $T$ be a trivalent tree with $n\geq 4$ leaves and rooted in $r$. Let $\hat{p}\in \cM_{T}$ be such that $\hat{\kappa}_{ij}=0$ for all $i,j\in [n]$ and $\omega_{0}\in \widehat{\Omega}_{0}$. Let $\delta=\delta(\omega_{0})$ be such that $\delta_{v}=1$ if $s_{v}^{0}=1$ and it is zero otherwise. Define $Q_{\delta} (\omega)$ as in (\ref{eq:Q1}). If either $\delta_{r}=0$ or $\delta_{r}=1$ and $\delta_{v}=1$ for all $(r,v)\in E$ then ${\rm mult}_{{0}}(Q_{\delta} )=1$. \end{lem}
\begin{proof}
A standard result for Minkowski sums says that each face of a Minkowski sum of two polyhedra can be decomposed as a sum of two faces of the summands and this decomposition is unique. Each facet of $\Gamma_{+}(Q_{\delta} )$ is decomposed as a face of the standard cone $\R^{n_{e}+|\delta|}_{\geq 0}\subset \cV_{\delta}$ plus a face of $\Gamma(Q_{\delta} )$. We say that a face of $\Gamma(Q_{\delta} )$ induces a facet of $\Gamma_{+}(Q_{\delta} )$ if there exists a face of the standard cone $\R^{n_{e}+|\delta|}_{\geq 0}$ such that the Minkowski sum of these two faces gives a facet of $\Gamma_{+}(Q_{\delta} )$. However, since the dimension $\Gamma(Q_{\delta} )$ is lower than the dimension of the resulting polyhedron it turns out that one face of $\Gamma(Q_{\delta} )$ can induce more than one facet of $\Gamma_{+}(Q_{\delta} )$. In particular $\Gamma(Q_{\delta} )$ itself induces more than one facet and one of them is $F_{0}$ given by (\ref{eq:F0}). 

Every facet of $\Gamma_{+}(Q_{\delta} )$ containing $4/n\mathbf{1}$ after normalizing the coefficients to sum to $n$, i.e. $\sum_{v}\a_{v}+\sum_{e}\beta_{e}=n$, is of the form 
\begin{equation}\label{eq:pomoc-facet}
\sum_{v}\a_{v}y_{v}+\sum_{e}\beta_{e}x_{e}\geq 4,
\end{equation} 
where by Lemma \ref{lem:positivecoef} we can assume that $\a_{v},\beta_{e}\geq 0$.

Our approach can be summarized as follows. Using Construction \ref{def:constr_net} we provide coordinates of a point $q\in\Gamma(Q_{\delta} )$ such that $4/n\mathbf{1}$ lies on the boundary of $q+\R^{n_{e}+|\delta|}_{\geq 0}$. Then $4/n \mathbf{1}$ can only lie on faces of $\Gamma_{+}(Q_{\delta} )$ induced by faces of $\Gamma(Q_{\delta} )$ containing $q$. 

First, assume that $\delta_{r}=0$ which corresponds to the case when the root $r$ represents a non-degenerate random variable. Consider the point $q\in\Gamma(Q_{\delta})$ induced by the network of $2n$ paths given in Construction \ref{def:constr_net}. Since $x_{e}=4/n$ for all $e\in E$ then from the description of $\Gamma(Q_{\delta} )$ in Lemma \ref{lem:structureG} we can check that all defining inequalities are strict for this point. Therefore $q$ lies in the interior of $\Gamma(Q_{\delta} )$. Therefore, the only facets of $\Gamma_{+}(Q_{\delta} )$ containing $q$ are these induced by $\Gamma(Q_{\delta} )$ itself. The equation defining a facet induced by $\Gamma(Q_{\delta})$ has to be obtained as a combination of the defining equations: $\sum_{e\in E_{0}}x_{e}=4$ and $|\delta|$ equations
\begin{equation}\label{eq:pomoc-ineq}
2y_{v}-x_{v\ch_{1}({v})}-x_{v\ch_{2}({v})}+x_{\pa({v})v}=0
\end{equation}
for all $v\in V$ such that $\delta_{v}=1$. We check possible combinations such that the form of the induced inequality as given in (\ref{eq:pomoc-facet}) is attained. The first inequality, defining $F_{0}$, is already of this form (c.f. equation (\ref{eq:F0})). The sum of all the coefficients is $n$ since there are $n$ terminal edges. Any other facet has to be obtained by adding to the first equation (since the right hand side in (\ref{eq:pomoc-facet}) is $4$) a non-negative (since the coefficients in front of $y_{v}$ need to be non-negative) combination of equations in (\ref{eq:pomoc-ineq}). However, since the sum of the coefficients in (\ref{eq:pomoc-ineq}) is $+1$ this contradicts the assumption that the sum of coefficients in the defining inequality is $n$. Consequently, if $\delta_{r}=0$ the codimension of the face hit by $4/n \mathbf{1}$ is $1$ and hence by Theorem \ref{thm:arnold} we have that ${\rm mult}_{{0}}(Q_{\delta} )=1$. 

Second, if $\delta_{r}=1$ and $\delta_{v}=1$ for all children of $r$ in $T$ then since all the nodes adjacent to $r$ (denote them by $a,b,c$) are inner we have three different ways of conducting the construction of the $n$-path network in Lemma \ref{def:constr_net} (by omitting each of the incident edges). Hence we get three different points and their barycenter satisfies $x_{ra}=x_{rb}=x_{rc}=8/3n$ and $x_{e}=4/n$ for all the other edges; $y_{r}=4/n$, $y_{a}=y_{b}=y_{c}=8/3n$ and $y_{v}=2/n$ for all the other inner nodes. Denote this point by $q$. By the facet description of $\Gamma(Q_{\delta} )$ derived in Proposition \ref{lem:structureG} we can check that this point cannot lie in any of the facets defining $\Gamma(Q_{\delta} )$ and hence it is an interior point of the polytope. As in the first case it means that the facets of $\Gamma_{+}(Q_{\delta} )$ containing $q$ are induced by  $\Gamma(Q_{\delta} )$. By Proposition \ref{lem:structureG} the affine span is given by the equation defining $F_{0}$, the equations (\ref{eq:pomoc-ineq}) for all inner edges $v$ apart from the root and in addition for the root we have
\begin{equation}\label{eq:pomoc-yr}
2y_{r}-x_{ra}-x_{rb}-x_{rc}=0.
\end{equation}
Since the sum of coefficients in the above equation is negative we cannot use the same argument as in the first case. Instead we add to $\sum_{e\in E_{0}}x_{e}=4$ a non-negative combination of equations in (\ref{eq:pomoc-ineq}) each with coefficient $t_{v}\geq 0$ and then add (\ref{eq:pomoc-yr}) with coefficient $\sum_{v\neq r} t_{v}$. The sum of coefficients in the resulting equation will be $n$ by construction. The coefficient of $x_{ra}$ is $t_{a}-\sum_{v\neq r}t_{v}$. Since it has to be non-negative it follows that $t_{v}=0$ for all $v$ apart from $a$. However, by checking the coefficient of $x_{rb}$ one deduces that $t_{v}=0$ for all inner nodes $v$. Consequently the only possible facet of $\Gamma_{+}(Q_{\delta} )$ containing $4/n\mathbf{1}$ is $F_{0}$ and hence again ${\rm mult}_{{0}}(Q_{\delta} )=1$. 

\end{proof}

\section{Proof of Theorem \ref{th:main}}\label{sec:last-proof}

In this section we complete the proof of Theorem \ref{th:main}. We split the proof into three steps. 
\paragraph{\textbf{Step 1}} To approximate $\log Z(N)$ it suffices to approximate $\log I(N)$, where $I(N)$ is given by (\ref{eq:ZN})
because $\log Z(N)=\hat{\ell}_{N}+\log I(N)$. By Theorem \ref{th:watanabe} equivalently we can compute ${\rm RLCT}_{\Theta_{T}}(f;\varphi)$, where $f$ is the normalized log-likelihood and $\varphi$ is the prior distribution satisfying (A1). By Theorem \ref{th:barI} and Theorem \ref{th:redtokappas} this real log-canonical threshold is equal to  ${\rm RLCT}_{\Omega_{T}}(\cI)$, where $\cI$ is the ideal given by (\ref{eq:pullback}).

\paragraph{\textbf{Step 2}} We compute separately ${\rm RLCT}_{\Omega_{T}}(\cI)$ in the case when $n=3$. If $T$ is rooted in the inner node the approximation for $\log Z(N)$ follows from Theorem 4 in \cite{rusakov2006ams}. Thus if $\widehat{E}=E$, which in \cite{rusakov2006ams} corresponds to the type 2 singularity, then 
\begin{equation}\label{eq:n3sing}
\log Z(N)=\hat{\ell}_{N} -2\log N+O(1) \quad \mbox{or} \qquad {\rm RLCT}_{\Omega_{T}}(\cI)=(2,1).
\end{equation}
Since all the neighbours of the root are leaves and hence by (A2) they are non-degenerate we need only to make sure that  the first equation in Theorem \ref{th:main} gives (\ref{eq:n3sing}). This follows from the fact that $l_{2}=0$ and $l_{3}=0$. In the case when $|\widehat{E}|=1$ (type 1 singularity) we have
$$
\log Z(N)=\hat{\ell}_{N} -\frac{5}{2}\log N+O(1)\quad \mbox{or} \qquad {\rm RLCT}_{\Omega_{T}}(\cI)=(\frac{5}{2},1).
$$
The second equation in Theorem \ref{th:main} holds since  $l_{2}=1$, $l_{3}=0$ and $c=0$. If $\widehat{E}=\emptyset$ we have
$$
\log Z(N)=\hat{\ell}_{N}-\frac{7}{2}\log N+O(1)\quad \mbox{or} \qquad {\rm RLCT}_{\Omega_{T}}(\cI)=(\frac{7}{2},1),
$$
which again is true since $l_{2}=0$, $l_{3}=1$ and $c=0$.

Now assume that $T$ is rooted in a leaf, say $1$. If there exists $i,j=1,2,3$ such that $\hat{\kappa}_{ij}\neq 0$ (or $|\widehat{E}|\leq 1$) then $\widehat{V}=\emptyset$ and by Proposition \ref{prop:reg}
$$
\log Z(N)=\hat{\ell}_{N}-\frac{7-2l_{2}}{2}+O(1)\quad \mbox{or} \qquad {\rm RLCT}_{\Omega_{T}}(\cI)=(\frac{7-2l_{2}}{2},1).
$$ 
If $\widehat{E}=E$ then by Theorem \ref{th:redtokappas} for every $\omega_{0}\in \widehat{\Omega}_{T}$
$$
{\rm RLCT}_{\omega_{0}}(\cI)=(\frac{3}{2},0)+{\rm RLCT}_{\omega_{0}}(\cJ).
$$
Moreover, by Lemma \ref{lem:Jjaksiepatrzy}, for every $\omega_{0}\in \widehat{\Omega}_{0}$
$$
{\rm RLCT}_{\omega_{0}}(\cJ)={\rm RLCT}_{{0}}(\la\eta_{1,h}\eta_{h,2}, \eta_{1,h}\eta_{h,3}, s_{h}^{\delta_{h}}\eta_{h,2}\eta_{h,3}\ra),
$$
where $\delta_{h}=1$ if $s_{h}^{0}=1$ and $\delta_{h}=1$ otherwise. It can be checked directly by using the Newton diagram method and  Theorem \ref{thm:arnold} that ${\rm RLCT}_{\omega_{0}}(\cJ)=(\frac{3}{4},1)$ both if $\delta_{h}=0$ and $\delta_{h}=1$ and hence ${\rm RLCT}_{\omega_{0}}(\cI)=(\frac{9}{4},1)$. Since the points in $\widehat{\Omega}_{0}$ such that $s_{h}^{0}\neq 1$ lie in the interior of $\Omega_{T}$ then for these points ${\rm RLCT}_{\omega_{0}}(\cI)={\rm RLCT}_{\Omega_{0}}(\cI)$ where $\Omega_{0}$ is a neighborhood of $\omega_{0}$ in $\Omega_{T}$. Hence by (\ref{eq:takemin}) we have that 
$${\rm RLCT}_{\Omega_{T}}(\cI)=\min_{\omega_{0}\in\widehat{\Omega}_{T}}{\rm RLCT}_{\Omega_{0}}(\cI)\leq \min_{\omega_{0}\in\widehat{\Omega}_{0}\cap{\rm int}(\Omega_{T})}{\rm RLCT}_{\Omega_{0}}(\cI)=(\frac{9}{4},1).$$ 
On the other hand by (\ref{eq:rlct-constr}) and then Proposition \ref{lem:red-2-deepest}
$${\rm RLCT}_{\Omega_{T}}(\cI)\geq \min_{\omega_{0}\in\widehat{\Omega}_{T}}{\rm RLCT}_{\omega_{0}}(\cI)\geq \min_{\omega_{0}\in\widehat{\Omega}_{{\rm deep}}}{\rm RLCT}_{\omega_{0}}(\cI)=(\frac{9}{4},1).$$
It follows that
\begin{equation}\label{eq:n3leafroot}
\log Z(N)=\hat{\ell}_{N}-\frac{9}{4}\log N+O(1)\quad \mbox{or} \qquad {\rm RLCT}_{\Omega_{T}}(\cI)=(\frac{9}{4},1),
\end{equation}
which gives the the second equation in Theorem \ref{th:main} since in this case $l_{2}=l_{3}=c=0$.

\paragraph{\textbf{Step 3, Case 1}} Assume now that $n\geq 4$ and $r\notin\widehat{V}$. In this case every $T_{i}$ for $i=1,\ldots,k$ is rooted in one of its leaves. Hence ${\rm rlct}_{\omega_{0}}(\cJ(T_{i}))=|L_{i}|/4$ for every $i=1,\ldots k$. If $|L_{i}|\neq 3$ this follows from Proposition \ref{prop:zerocase}. If $|L_{i}|=3$ it follows from Case 2 above. By Lemma \ref{lem:redzQ} and Proposition \ref{prop:reg}, for every $\omega_{0}\in \widehat{\Omega}_{0}$ we have that
$$
{\rm rlct}_{\omega_{0}}(\cI)=(\frac{n}{2},0)+\sum_{i=1}^{m} \frac{n_{v}^{i}+n_{e}^{i}-n_{i}-2l_{2}^{i}}{2}+\sum_{i=1}^{k} \frac{|L_{i}|}{4},
$$
where $n_{v}^{i}$, $n_{e}^{i}$, $l_{2}^{i}$ are respectively the number of vertices, edges and and degree two nodes in $\widehat{T}$ of $S_{i}$; and $L_{i}$ is the set of leaves of $T_{i}$. We use three simple formulas: $\sum_{i} n_{v}^{i}=l_{1}+l_{2}+l_{3}$ (i.e. only degree zero nodes of $\widehat{T}$ do not lie in the $S_{i}$'s), $\sum_{i}n_{e}^{i}=|E\setminus \widehat{E}|$ (i.e. $E\setminus\widehat{E}$ is the set of all edges of all the $S_{i}$'s) and $\sum_{i}|L_{i}|=l_{2}+n-l_{1}$ (i.e. the leaves of all the $T_{i}$'s are precisely the degree two nodes of $\widehat{T}$ and these leaves of $T$ which have degree zero). Moreover for any graph with the vertex set $V$ and the edge set $E$, $\sum_{v\in V} \deg(v)=2{n_{e}}$ (see e.g. Corollary 1.2.2 in \cite{semple2003pol}). Therefore with the formula applied for $\widehat{T}$ we have $l_{1} +2l_{2}+3l_{3}=2|E\setminus \widehat{E}|$. Using these four formulas we show that ${\rm rlct}_{\omega_{{0}}}(\cI)=\frac{1}{4}(3n+l_{2}+5l_{3})$. Moreover, since $\delta_{r}=0$ for all $\omega_{0}\in\widehat{\Omega}_{0}$ then by Lemma \ref{lem:multiplicities} ${\rm mult}_{{0}}(\cJ(T_{i}))=1$ for every $\omega_{0}\in\widehat{\Omega}_{0}$. Therefore, 
\begin{equation}\label{eq:fina-in-fina-proof}
{\rm RLCT}_{\omega_{{0}}}(\cI)=\left(\frac{1}{4}(3n+l_{2}+5l_{3}),1\right).
\end{equation}
Now we show that ${\rm RLCT}_{\Omega_{T}}(\cI)$ also has the same form. Let ${\omega}_{2}$ be a point in $\widehat{\Omega}_{0}$ such that $s_{v}\neq 1$ for all $v\in V$ and let $\omega_{1}\in \widehat{\Omega}_{{\rm deep}}$. Equation (\ref{eq:fina-in-fina-proof}) is true both if $\omega_{0}=\omega_{1}$ and $\omega_{0}=\omega_{2}$ and hence ${\rm RLCT}_{\omega_{1}}(\cI)={\rm RLCT}_{\omega_{2}}(\cI)$.
However, since ${\omega}_{2}$ is an inner point of $\Omega_{T}$ it follows, from the definition of ${\rm RLCT}_{\Omega_{T}}(\cI)$ as the minimum over all points in $\Omega_{T}$, that 
$${\rm RLCT}_{\Omega_{T}}(\cI)\leq {\rm RLCT}_{{0}}(\cI_{{\omega}_{2}}).$$ 
On the other hand by (\ref{eq:rlct-constr}) and Proposition \ref{lem:red-2-deepest}
$$
{\rm RLCT}_{0}(\cI_{\omega_{1}})=\min_{\omega_{0}\in \widehat{\Omega}_{T}} {\rm RLCT}_{0}(\cI_{\omega_{0}})\leq \min_{\omega_{0}\in \widehat{\Omega}_{T}} {\rm RLCT}_{\Omega_{0}}(\cI_{\omega_{0}})={\rm RLCT}_{\Omega_{T}}(\cI).
$$
Therefore, if $r\notin\widehat{V}$ then in fact ${\rm RLCT}_{\Omega_{T}}(\cI)=(\frac{1}{4}(3n+l_{2}+5l_{3}),1)$ and hence 
$$
Z(N)=\hat{\ell}_{N}-\frac{1}{4}\left(3n+l_{2}+5l_{3}\right)\log N + O(1).
$$
Hence in this case the main formula in Theorem \ref{th:main} is proved since $c=0$.

\paragraph{\textbf{Step 3, Case 2}} Let now $n\geq 4$ and $r\in\widehat{V}$. Let $1\leq j\leq k$ be such that $r$ is an inner node of $T_{j}$ and $\omega_{0}\in \widehat{\Omega}_{0}$. For all $i\neq j$ $T_{i}$ is rooted in one of its leaves. Therefore, by Lemma \ref{lem:lambda}, Lemma \ref{lem:multiplicities} and Step 2 above for all $i\neq j$ we have that ${\rm RLCT}_{\omega_{0}}(\cJ(T_{i}))=(|L_{i}|/4,1)$.  It remains to compute ${\rm RLCT}_{\omega_{0}}(\cJ(T_{j}))$.  If $|L_{j}|=3$ then ${\rm RLCT}_{\omega_{0}}(\cJ(T_{j}))=(1/2,1)=((|L_{j}|-1)/4,1)$ by the Step 2 above (c.f.  (\ref{eq:n3sing})). Therefore in this case the computations are the same as in Step 3, Case 1 but with a difference of $\frac{1}{4}$ in the real log-canonical threshold.  Therefore we obtain
$$
Z(N)=\hat{\ell}_{N}-\frac{1}{4}\left(3n+l_{2}+5l_{3}-1\right)\log N + O(1).
$$

However, if $|L_{j}|\geq 4$ then by Lemma \ref{lem:lambda} ${\rm rlct}_{0}(\cJ(T_{j}))=|L_{j}|/4$ and hence as in Step 3, Case 1 we have $\sum_{i=1}^{k} {\rm rlct}_{0}(\cJ_{\omega_{0}}(T_{i}))=\frac{1}{4}(n-l_{1}+l_{2})$. Therefore
$$
{\rm rlct}_{\Omega_{T}}(\cI)=\frac{1}{4}(3n+l_{2}+5l_{3}).
$$
We compute the multiplicity by considering different subcases. If all the neighbours of $r$ are degenerate then for all points $\omega_{0}\in \widehat{\Omega}_{{\rm deep}}$ we have that $\delta_{r}=1$ and $\delta_{v}=1$ for all neighbours $v$ or $r$. It follows from Lemma \ref{lem:multiplicities} that ${\rm mult}_{\omega_{0}}(\cJ(T_{j}))=1$ and hence ${\rm mult}_{\Omega_{T}}(\cI)=1$. Therefore, 
$$
Z(N)=\hat{\ell}_{N}-\frac{1}{4}\left(3n+l_{2}+5l_{3}\right)\log N + O(1).$$
Otherwise we do not have explicit bounds on the multiplicity. Since ${\rm mult}_{\Omega_{T}}(\cI)\geq 1$ then 
$$Z(N)=\hat{\ell}_{N}-\frac{1}{4}\left(3n+l_{2}+5l_{3}\right)\log N + (m-1)\log\log N+O(1),
$$ 
where $m\geq 1$. This finishes the proof of Theorem \ref{th:main}.\hfill
$\qed$

%
%

\section*{Acknowledgments}

I am especially grateful to Shaowei Lin for a number of helpful comments and for introducing me to the theory of log-canonical thresholds. I also want to thank Diane Maclagan for helpful discussions.


\bibliographystyle{siam}
\bibliography{../!bibliografie/algebraic_statistics}
\end{document}